\documentclass[11pt, article, reqno]{amsart}

\allowdisplaybreaks

\usepackage[margin=16mm]{geometry}
\usepackage{amssymb, amsmath, amsthm, amscd}

\usepackage[T1]{fontenc}
\usepackage{eucal,mathrsfs}
\usepackage{color}

\renewcommand{\leq}{\leqslant}

\renewcommand{\le}{\leqslant}
\renewcommand{\ge}{\geqslant}

\def\eps{\varepsilon}

\definecolor{mno}{rgb}{0.5,0.1,0.5}

\newcommand{\R}{\mathbb R}

\newcommand{\e}{\varepsilon}
\newcommand{\T}{\mathbb T}

\newcommand{\Pp}{\mathbb P}
\newcommand{\Ee}{\mathbb E}

\newcommand{\I}{\mathbf 1}
\newcommand{\w}{\omega}
\newcommand{\N}{\mathbb{N}}
\newcommand{\Z}{\mathbb Z}
\newcommand{\sL}{\mathcal{L}}

\newcommand{\E}{\mathscr{E}}
\newcommand{\F}{\mathscr{F}}

\newcommand{\LL}{\mathcal{L}}

\newtheorem{theorem}{Theorem}[section]
\newtheorem{lemma}[theorem]{Lemma}
\newtheorem{proposition}[theorem]{Proposition}

\theoremstyle{definition}

\newtheorem{remark}[theorem]{Remark}

\numberwithin{equation}{section}

\begin{document}

\title[Quantitative stochastic homogenization for stable-like jumps] {Quantitative stochastic homogenization
for random
conductance models with stable-like jumps}

\author{Xin Chen,\quad Zhen-Qing Chen,\quad Takashi Kumagai  \quad \hbox{and}\quad Jian Wang}

\date{}

\maketitle

\begin{abstract} We consider random conductance
models with long range jumps on $\Z^d$, where
the one-step transition probability from $x$ to $y$ is proportional to
$w_{x,y}|x-y|^{-d-\alpha}$ with $\alpha\in (0,2)$. Assume that $\{w_{x,y}\}_{(x,y)\in E}$ are independent, identically distributed and uniformly bounded
non-negative random variables with
 $\Ee w_{x,y}=1$, where $E$ is the set of all
unordered pairs on $\Z^d$. We obtain a quantitative version of stochastic homogenization for these random walks, with
explicit polynomial rates  up to logarithmic corrections.

\medskip


\medskip

\end{abstract}
\allowdisplaybreaks

\section{Introduction and main result}\label{section1}

Some advances have been made recently on
stochastic homogenization for non-local operators with random coefficients; see \cite{BCKW, CCKW2, CKK, FH, FHS, PZ}. In particular, the qualitative homogenization results established in \cite{CCKW2} (adjusted to a discrete setting) can be roughly restated as follows. Consider the following
fractional Laplacian-like (or $\alpha$-stable-like)
random operator
$$
\sL^\w f(x)=
\sum_{y\in \Z^d: y\not=x}
\left(f(y)-f(x)\right)\frac{w_{x,y}(\w)}{|x-y|^{d+\alpha}},
$$
 where $\alpha\in (0,2)$,
$w_{x,y}(\w)$ is a non-negative measurable function defined on $\Z^d \times \Z^d \times \Omega$ so that $w_{x,y}(\w)=w_{y,x}(\w)$,  and $\w \in \Omega$
 describes a stationary and ergodic environment. For $\varepsilon>0$, let $\sL^{\e,\w}$ be the associated scaled operator
$$
\sL^{\e,\w} f(x)=\eps^d \sum_{y\in \e \Z^d: y\not=x}
 \left(f(y)-f(x)\right)\frac{w_{x/\e,y/\e}(\w)}{|x-y|^{d+\alpha}}.
$$
As a special case of the main results established in \cite{CCKW2},
we know that, under  some suitable conditions,  for
any $\lambda>0$ and $f\in C_c(\R^d)$, the unique solution
$u^{\w}_\e$ of
$$(\lambda -\sL^{\e,\w})
u^{\w}_\e=f$$
converges in $L^2$,
as $\e\to0$, to the solution
$\bar u$
of the deterministic equation
$$(\lambda -\bar \sL) \bar u=f,$$ where $\bar \sL$ is an $\alpha$-stable-like operator with deterministic coefficients.
The goal of this paper is to
investigate
the quantitative stochastic homogenization; namely,
to find the rate of
$L^2$-convergence of $u_\e$ to $\bar u$ for $f$ in a dense subset of $L^2(\R^d; dx)$.

Stochastic homogenization has been intensively studied for a long time,
mostly for differential operators and for finite range random walks  in random environments,
 since
the pioneering works of Kozlov \cite{Ko} and Papaicolaou and Varadhan \cite{PV}
for
divergence form differential operators
  with random coefficients. Much efforts have been
focused on qualitative results such as
proving the existence of a homogenized equation
that
 characterizes the limit. See \cite{ABDH, ACJS, ACS, B,Ku} and references therein.

The first quantitative
result
of stochastic homogenization for divergence form differential operators
was obtained by Yuinskii \cite{Yu}.
Recently, Gloria and Otto \cite{Go1, Go2} obtained the optimal variance estimate and error estimates  (in terms of the ratio of length scales) for the correctors in the discrete and uniformly
elliptic setting,
by combining
elliptic and  parabolic regularity theory and Green function estimates with a spectral gap inequality or a logarithmic Sobolev inequality for
vertical derivative of environments, which
are
used to quantify the ergodicity.
Based on these estimates, the quantification for this discrete model and the counterpart in the continuous configuration space
were
investigated by Gloria, Neukamm and Otto \cite{GNO} and Gloria and Otto \cite{Go3}, respectively.

The other point of view developed by Armstrong, Smart and Mourrat (see \cite{AKM1, AM, AS} or the book \cite{AKM}) is to establish the quantitative theory of homogenization by focusing on the monotonicity
and subadditivity
of certain energy quantities, which implements a progressive coarsening of the coefficient field and
captures the behavior of solutions on large length scales.
Besides  the monotonicity property
and subadditivity property, the large scale regularity for elliptic equation as well as
some bootstrap arguments were adopted to prove the optimal convergence rate for
elliptic homogenization; see e.g. \cite{AKM, AKM1, AKM2}.
We refer the reader to Armstrong, Bordas and Mourrat \cite{ABM} and Gu and Mourrat \cite{GM2}
for the quantitative results in the parabolic setting.

Armstrong and Dario \cite{AD} have
obtained the quantitative results for the homogenization problems for simple random walks on a supercritical Bernoulli percolation cluster with possibly degenerate conductances. Later the
quantification of  local central limit theorem for this model has been established by Dario and Gu \cite{DG}. See Giunti and Mourrat \cite{GM} for the quantification
in degenerate environments under
more general inverse moment conditions.

However, there is no result on the quantitative theory of homogenization for non-local $\alpha$-stable-like operators in the literature
(even for the deterministic periodic environments).
 There are significant differences
between the quantitative homogenization  theory for elliptic operators of  divergence form  and
that for symmetric
$\alpha$-stable-like
non-local operators.
The main
reason
is that,  due to the
heavy tail of the jumping kernels,
 the method
 via a global corrector for the corresponding elliptic
 equation,  which is necessary for the results of nearest neighbor models and elliptic homogenization quoted above,
  does not work for non-local $\alpha$-stable-like operators.
Moreover, because of
the long range interactions,
neither
functional inequalities for the vertical derivative of environments nor
the monotonicity of the energy functionals mentioned above can be applied to long range conductance models.

\medskip

The aim of this paper is to establish quantitative results for homogenization problems of non-local $\alpha$-stable-like operators.
There are two main novel
 ingredients of our approach.
One is the multi-scale Poincar\'e inequality for non-local Dirichlet forms with
possibly
degenerate coefficients (see Proposition \ref{l2-2}), and the other one is the construction and the local $H^{\alpha/2}$-estimates of correctors via local Poisson equations \eqref{e2-1} and \eqref{e2-1-2--} (instead of global correctors for nearest neighbor models and elliptic homogenization as stated above). To the best of
the authors' knowledge, this is the first paper on
quantitative results for homogenization of non-local operators.
Below we describe the setting and the main result of this paper.

Throughout the paper, we always suppose that the following assumption holds.

\medskip

\noindent{{\bf Assumption (H1)}}: {\it Let $E=\{(x, y):x,y\in \Z^d\}$ be the collection of all unordered
distinct pairs
on $\Z^d$.
Suppose that $\{w_{x,y}: (x,y)\in E\}$ is a sequence of i.i.d. non-negative random variables such that the following
hold{\rm:}
\begin{itemize}
\item [(i)] $w_{x,y}=w_{y,x}$ for
every $x\not= y\in \Z^d$;

\medskip
\item [(ii)] $\Ee[w_{x,y}]=1$ for
every $x\not= y\in \Z^d$;

\medskip
\item [(iii)] There exists a  constant $C_1>0$ such that
$w_{x,y}\le C_1$ for all $x\not= y\in \Z^d$.
\end{itemize}}

\medskip

Note that under {\bf Assumption (H1)},  the random
variables
$w_{x,y}$ may
degenerate.

\medskip

Let $\mu$ be the counting measure on $\Z^d$. For $\omega\in \Omega$,
let $(\mathscr{E}^\w, \mathscr{F}^\w)$ be the following Dirichlet form on $L^2(\Z^d;\mu)$:
\begin{align*}
\mathscr{E}^\w(f,f)=&\frac{1}{2}
\sum_{x,  y\in \Z^d: x\not= y}
\left(f(x)-f(y)\right)^2\frac{w_{x,y}(\w)}{|x-y|^{d+\alpha}},\\
\mathscr{F}^\w=&\{f\in L^2(\Z^d;\mu): \mathscr{E}^\w (f,f)<\infty\},
\end{align*}
where $\alpha\in (0,2)$. Note that, under {\bf Assumption (H1)}(iii), it is easy to see $\mathscr{F}^\w= L^2(\Z^d;\mu)$ for
every $\omega\in \Omega$. Furthermore, it follows from the proof of \cite[Theorem 3.2]{CKK} that the Dirichlet form $(\mathscr{E}^\w, \mathscr{F}^\w)$ is regular with the core $B_c(\Z^d)$, the set of
functions on $\Z^d$ with compact support.
Let $(X_t^\omega)_{t\ge 0}$ be
the continuous time $\Z^d$-valued symmetric Hunt   process
  associated with the regular Dirichlet form $(\mathscr{E}^\w, \mathscr{F}^\w)$
  on $L^2(\Z^d; \mu)$.  The
  infinitesimal generator $\sL^\w$ associated with
$(\mathscr{E}^\w, \mathscr{F}^\w)$ is given by
$$
\sL^\w f(x)=
 \sum_{y\in \Z^d:y\not = x}
\left(f(y)-f(x)\right)\frac{w_{x,y}(\w)}{|x-y|^{d+\alpha}},\quad f\in L^2(\Z^d;\mu).
$$
For any $k\in \N_+:=\{1,2,\cdots\}$, let $\mu^k$ be the normalized counting measure on $k^{-1}\Z^d$ so that
$\mu^k ((0, 1]^d)=1$ i.e., for any $A\subset k^{-1}\Z^d$, $\mu^k(A):=
k^{-d}\sum_{x\in k^{-1} \Z^d}1_A(x)$.
Consider the following
scaled operator $\sL^{k,\w}$ of the operator $\sL^{\w}$:
\begin{equation}\label{e:operator}
\sL^{k,\w} f(x):=k^{-d}
  \sum_{y\in k^{-1}\Z^d: y\not= x}
\left(f(y)-f(x)\right)\frac{w_{kx,ky}(\w)}{|x-y|^{d+\alpha}},\quad x\in k^{-1}\Z^d,\
f\in L^2(k^{-1}\Z^d;\mu^k).
\end{equation}
We sometimes
drop
$\w$ and simply write $\sL^{k}$ for $\sL^{k,\w}$.
It is standard to check that $\sL^{k,\w} $ is the infinitesimal operator for the $k^{-1}\Z^d$-valued process
$\{X_t^{k\,\w}\}_{t\ge 0}:=\{k^{-1}
X_{k^\alpha t}^{\w}\}_{t\ge 0}$,
and that for any
$f, g \in L^2(k^{-1}\Z^d;\mu^k)$,
\begin{equation}\label{e1-1a}
-\int_{k^{-1}\Z^d}f(x)\sL^{k,\w}g(x)\,\mu^k(dx)=\frac{1}{2}
\int_{k^{-1}\Z^d}\int_{k^{-1}\Z^d}
(f(x)-f(y)) (g(x)- g(y))\frac{w_{kx,ky}(\w)}{|x-y|^{d+\alpha}}\,\mu^k(dx)\,\mu^k(dy).
\end{equation}

We will  estimate the speed of
the $L^2$-convergence of the $\lambda$-resolvent
  of $ \sL^{k,\w}$ as $k\to \infty$ to that
of the limit
homogenized generator  $\bar \sL$, defined by
\begin{eqnarray}
\bar \sL f(x)&:=&\lim_{\varepsilon \to 0}\int_{\{|y-x|\ge \varepsilon\}} \left(f(y)-f(x)\right)\frac{1}{|x-y|^{d+\alpha}}\,dy
\quad \hbox{for } f\in C^2_c (\R^d)\nonumber\\
&=&\int_{\R^d} \left(f(x+z)-f(x)-\nabla f(x)\cdot z \I_{\{|z|\le 1\}}\right)\frac{1}{|x-y|^{d+\alpha}}\,dy, \\
&=&\begin{cases}  \displaystyle\int_{\R^d} \left(f(x+z)-f(x) \right)\frac{1}{|x-y|^{d+\alpha}}\,dy,\quad &\hbox{when }\alpha\in (0,1),\\
\displaystyle\int_{\R^d} \left(f(x+z)-f(x)-\nabla f(x)\cdot z \I_{\{|z|\le 1\}}\right)\frac{1}{|x-y|^{d+\alpha}}\,dy,\quad
&\hbox{when } \alpha = 1, \\
\displaystyle\int_{\R^d} \left(f(x+z)-f(x)-\nabla f(x)\cdot z \right)\frac{1}{|x-y|^{d+\alpha}}\,dy,\quad
&\hbox{when } \alpha\in (1,2).
\end{cases}\label{e:1.2}
\end{eqnarray}

Note that $\bar \sL$ is a constant multiple of the classical fractional Laplacian $\Delta^{\alpha/2}:= - (-\Delta)^{\alpha/2}$.
For $\lambda >0$, let
$$
\mathscr{S}_0^\lambda:= \left\{f:  f= (\lambda - \bar \sL) g \hbox{ for some } g\in   C_c^\infty(\R^d) \right\}.
$$
In other words, if we denote the $\lambda$-resolvent operator for the fractional Laplacian $\bar \sL$
by $\bar R_\lambda$, then
$$
\mathscr{S}_0^\lambda =\{f:  \bar R_\lambda f \in C_c^\infty (\R^d)\}.
$$
It is easy to see that
$\mathscr{S}_0^\lambda \subset C_b(\R^d)\cap L^2(\R^d;dx)\cap L^2(k^{-1}\Z^d;\mu^k)$ for every $k\ge 1$.
As explained at the beginning of \cite[Section 3]{PZ}, $\mathscr{S}_0^\lambda$ is dense in $L^2(\R^d;dx)$ (see
the appendix for
details)
and has been frequently used
in the study of
the homogenization problem in
general ergodic random media.

\medskip

 For any $f\in \mathscr{S}_0^\lambda$, let
 $u_k^{\w}\in L^2(k^{-1}\Z^d;\mu^k)$
be the unique weak
solution of the following scaled resolvent equation
\begin{equation}\label{e3-2}
(\lambda -\sL^{k,\w }) u^{\w}_k(x)=f(x)
\quad \hbox{for } x\in k^{-1}\Z^d.
\end{equation}
That is, $u^{\w}_k$ is the $\lambda$-resolvent of $\sL^{k,\w }$ for the function $f$ in $L^2$-space, which we denote
 by $R_\lambda^{k, \w}f$.
With a little abuse of notation, we
extend $u^{\w}_k$ to $u^{\w}_k: \R^d \to \R$ by setting $u^{\w}_k(x)=u^{\w}_k(z)$ if $x\in \prod_{1\le i\le d}(z_i,z_i+k^{-1}]$ for
some unique $z:=(z_1,z_2,\cdots, z_d)\in k^{-1}\Z^d$.

\medskip
The following is the main result of this paper.

\begin{theorem}\label{T:main}
Assume that {\bf Assumption (H1)} holds and $d>\alpha$.
Let $\lambda >0$ and $\gamma >0$. Then for any $f\in \mathscr{S}_0^\lambda$,
there are a constant $C_0>0$
{\rm (}which depends on $f$, $\gamma$, $\alpha$ and $\lambda${\rm )}
and a random variable
$k_0(\w)\ge1$ {\rm (}which depends on $\gamma$, $\alpha$ and $\lambda$ but is independent of $f${\rm )}
such that for all $k>k_0(\w)$,
\begin{equation}\label{t3-1-1-1}
\|
R_\lambda^{k, \w}f - \bar R_\lambda f
\|_{L^2(\R^d;dx)} \le C_0
\begin{cases} \min\left\{k^{-\alpha/2}
\log^{\frac{\alpha(1+\gamma)}{(4(d-\alpha))\wedge(2d)}}k,
\, k^{- ((1-\alpha)\wedge (d/2))}
\log^{1/2}k\right\},&\quad \alpha\in (0,1),\\
k^{-1/2}
\log^{1+\frac{1+\gamma}{(4(d-1))\wedge (2d)}}k,
&\quad \alpha=1,\\
k^{-(2-\alpha )/2}
 \log^{\frac{\alpha(1+\gamma)}{(4(d-\alpha))\wedge (2d)}} k,
 &\quad \alpha\in (1,2).\end{cases}
\end{equation}
\end{theorem}

\begin{remark}\label{R:1.2}\begin{itemize}
\item[(1)]
Since $R_\lambda^{k, \w}$  and $\bar R_\lambda$ are bounded operators on $L^2(\R^d; dx)$ with operator norms no larger than $1/\lambda$ and
$\mathscr{S}_0^\lambda$ is dense in $L^2(\R^d;dx)$, Theorem \ref{T:main} in particular implies the strong
resolvent convergence of $R_\lambda^{k, \w}$ to $\bar R_\lambda$  in $L^2$-sense, that is,
that for each $f\in L^2(\R^d; dx)$, $R_\lambda^{k, \w}f $ converges to $\bar R_\lambda f$ in $L^2(\R^d; dx)$.
The latter property is equivalent to the Mosco convergence of the
Dirichlet form associated with the scaled operator $\sL^{k,\w}$ to that of $\bar \sL$.
Hence Theorem \ref{T:main}  not only recovers the main result of \cite[Section 7]{CKK} on Mosco convergence of Dirichlet forms in random environment but also provies a rate of convergence for functions in a subspace  $\mathscr{S}_0^\lambda$ of $L^2(\R^d; dx)$.

\smallskip

\item[(2)]
 In the present setting, when $\alpha\in[1,2)$,
our rate of convergence is of order $(2-\alpha)/2$.
In  the deterministic periodic media case as studied in \cite{CCKW-new} or
 in the case of an i.i.d. sequence of random variables in the domain of attraction of a stable random variable
 studied in \cite{Ban} (see also \cite{CNX,X} for related topics),
the rate of convergence is of order $2-\alpha$.
 Indeed, in the periodic setting, we can construct a global corrector
$\phi\in C_b^1(\T^d)$
 so that
$\|\phi\|_\infty+\|\nabla \phi\|_\infty<\infty$;
while for
stable-like random walks on a mutually independent random conductance model, it seems
 there does not exist
such global corrector, due to the lack of $L^2$-integrability for the associated jumping kernel.
According to this significant distinction, it is quite plausible that for
stable-like  long range  random walks,  the convergence rate for the random conductance models
with mutually independent conductances is different from that of deterministic periodic
conductance models. More precisely,
when $\alpha\in (1,2)$, the convergence rate of
order
$k^{-(2-\alpha )/2}$ seems to be
reasonable in the sense that, the $H^{-\alpha/2}$ norm of the truncated potential $V_k^{\w}(x):=\sum_{z\in \Z^d:|z|\le k}{w_{x,x+z}(\w)}z{|z|^{-d-\alpha}}$ satisfies
$$\|V_k^\w\|_{H^{-\alpha/2}}^2\simeq
\sum_{z\in \Z^d:|z|\le k}{|z|^2}{|z|^{-d-\alpha}}\simeq k^{2-\alpha}.
$$
Intuitively, this blow up rate of $\|V^\w _k\|_{H^{-\alpha/2}}$
would slow down
the convergence rate from $k^{-(2-\alpha )}$
into $k^{-(2-\alpha )/2}$  with some logarithm corrections.
Note that for the nearest neighbor random walks (that is, for the deterministic constant conductance model)
 the limit of $\|V^\w _k\|_{H^{-\alpha/2}}$ exists as a finite number.
This property makes
the difference between  the  periodic setting and the mutually independent random conductance model.
We point out  that, when $\alpha\in (0,1)$,   we can obtain the convergence rate for the long range random conductance model
without using the corrector method; see the proof of Theorem \ref{T:main}.
We note that our convergence rate for $\alpha$
near $0$ and $2$,
modulo the logarithmic correction term,
 is half the distance of $\alpha$ to the boundary of $(0, 2)$.

\smallskip

\item [(3)]
We emphasize that
in the course of proving Theorem \ref{T:main},
a
Poincar\'e inequality for long range conductance models is obtained under
a very mild non-degenerate condition that
 $\Pp(w_{x,y}>0)>0$ for all $x,y\in \Z^d$; see Proposition \ref{l2-1} below.
 In particular, Proposition \ref{l2-1} improves  \cite[Lemma 2.2]{CKW21}, which is one of
the crucial ingredients
 for the quenched invariance principle for
 long range random conductance models
 in \cite{CKW21} and the associated heat kernel estimates in \cite{CKW20}.
  We point out that this minimal non-degenerate condition is much weaker than
	that
	is needed for
  nearest neighbor random conductance models.
 For example,
for invariance principle of random walks on supercritical percolation clusters, it requires that
 $\Pp(w_{x,y}>0)>p_c(d)$ for all $x,y\in \Z^d$ with $|x-y|=1$, where
  $p_c(d)\in (0,1)$
    is the critical bond percolation probability for $\Z^d$, that is, $p_c(d)$ is the infimum of  probability $p>0$ so that after removing each edge in $\Z^d$ independently with
  probability $1-p$,  there exists an  infinite connected  cluster almost sure; see,  e.g. \cite{Bar}.

\smallskip

\item [(4)]
Let us
mention some crucial steps
  to establish the convergence rate for $\alpha\in (1,2)$.
As explained before, one key ingredient in this paper is the energy estimate \eqref{p2-1-2} for the local
corrector $\phi_m$, which is defined by the solution
 of
 the Poisson equation \eqref{e2-1}. Combining this with the local Poincar\'e inequality, we can further obtain
the $L^2$-estimate \eqref{t3-1-2} for the local corrector $\phi_m$.
Note that, according to \eqref{e2-1}, we have the following expression for the local corrector $\phi_m$:
\begin{align*}
\phi_m(x)=-\sum_{z\in \Z^d\cap B_{2^m}}G_{B_{2^m}}(x,z)\left(V_m(z)-\oint_{B_{2^m}}V_m\,d\mu\right),\quad x\in B_{2^m}\cap \Z^d,
\end{align*}
where $G_{B_{2^m}}(\cdot,\cdot): B_{2^m}\cap \Z^d \times B_{2^m}\cap \Z^d \to \R_+$ denotes the Green function associated with the operator
$\LL_{B_{2^m}}^\w$ defined by \eqref{e1-1a}. When the conductance $\{w_{x,y}: (x,y)\in E\}$ is uniformly elliptic, it is easily seen from heat kernel estimates of symmetric stable-like processes (e.g., see \cite{CK08}) that
$$\sup_{x\in \Z^d\cap B_{2^m}}\sum_{z\in \Z^d\cap B_{2^m}}G_{B_{2^m}}(x,z)\le c_02^{m\alpha },$$
which, along with the expression for $\phi_m$ above, yields that
\begin{align*}
\oint_{B_{2^m}}|\phi_m(x)|^2\,\mu(dx)\le \|\phi_m\|_\infty^2 \le 4\|V_m\|_\infty^2\left(\sup_{x\in \Z^d\cap B_{2^m}}\sum_{z\in \Z^d\cap B_{2^m}}G_{B_{2^m}}(x,z)\right)^2
\le c_12^{2m\alpha}.
\end{align*}
Such estimate
is weaker than the desired assertion \eqref{t3-1-2}, and is not enough
even to show
$$\lim_{k \to \infty}\|u_k^\e-\bar u\|_{L^2(\R^d;dx)}=0.
 $$
This illustrates the crucial role played by the energy estimate \eqref{p2-1-2} in our investigation of quantitative homogenization,
which takes into account the quantification of ergodicity for long range random conductance models.

\smallskip

 \item [(5)]
 Our approach for the proof of
 Theorem \ref{T:main} works even when the uniformly bounded condition of $w_{x,y}$ in  {\bf Assumption (H1)}(iii) is replaced with some moment conditions on $w_{x,y}$. In this case, the convergence rates for $\|u^{\w}_k-\bar u\|_{L^2(\R^d;dx)}$ are still of polynomial decay but with worse
exponents. Similarly, when $w_{x,y}$ is uniformly bounded from below by a positive constant, the convergence rates for $\|u^{\w}_k-\bar u\|_{L^2(\R^d;dx)}$ can be improved a little bit by replacing
the
exponent
${\frac{\alpha(1+\gamma)}{(4(d-\alpha))\wedge(2d)}} $ in the logarithmic
correction term
by $ {\frac{\alpha(1+\gamma)}{4(d-\alpha)}}$,
see e.g. Remark \ref{e:order} below.

\smallskip

\item [(6)] For simplicity in the present paper
we only consider
the resolvent equations on the whole Euclidean space.
As investigated in \cite{ABC,RS,ZZ}, boundary regularity for the solution of the Dirichlet problem associated with
stable-like operators is totally different from that
of
elliptic second order differential operators.
We will
study the
quantitative
homogenization for solutions of the Dirichlet problem associated with stable-like operators
in deterministic periodic medium and in space-time random ergodic medium
in two separate papers \cite{CCKW-new,CCKW-new1}.
\end{itemize}

\end{remark}

The rest of this paper is
organized
as follows. In Section \ref{section2}, we establish
a Poincar\'e inequality and a multi-scale Poincar\'e inequality
for the Dirichlet form $(\mathscr{E}^\w, \mathscr{F}^\w)$.
In Section \ref{section2-}, we consider local energy estimates (i.e., $H^{\alpha/2}$-bounds) for
solutions of localized Poisson equations associated with the operator $\sL^{k,\w}$.
 In Section \ref{section3} we present uniform $L^2$-estimates for the differences between scaled operators $\sL^{k,\w}$, as well as their variants, and the limit operator $\bar \sL$. The last section is devoted to the proof of Theorem \ref{T:main}.

\medskip

\noindent{\bf Notation} Throughout this paper,
we use := as a way of definition.  For
 each $R>0$ and $x\in \R^d$, set $B_R(x):=x+(-R,R]^d$, and
with a little abuse of notation but it should be clear from the context,
we also use $B_R(x)$ to denote
$B_R(x) \cap \Z^d$ or $B_R(x) \cap  (k^{-1}\Z^d)$
for $k\ge 1$. For simplicity, we write $B_R=B_R(0)$.
For any $k\in \N$, let $\mu^k$ be the normalized counting measure on
$k^{-1}\Z^d$  so that $\mu^k ((0, 1]^d)=1$.
For $f: k^{-1}\Z^d \to \R$, set \begin{equation}\label{e:cou-k}\mu^k(f):=\int_{k^{-1}\Z^d} f(x)\,\mu^k(dx)=k^{-d}\sum_{x\in k^{-1}\Z^d} f(x).\end{equation}
For any subset $U\subset \Z^d$ and any
$\R^d$-valued
function $f: U \to \R^n$ with some $n\ge 1$, define
\begin{equation}\label{e1-1a}
\oint_U f\,d\mu=\oint_U f(x)\mu(dx)
:=\frac{1}{\mu(U)}\sum_{x\in U} f(x),     \quad
\sL_U^{\w} f(x):=\sum_{y\in U: y\neq x}\left(f(y)-f(x)\right)\frac{w_{x,y}(\w)}{|x-y|^{d+\alpha}}
\   \hbox{ for } x\in U,
\end{equation}
and
$$
\mathscr{E}_U^{\w}(f,f):=\sum_{i=1}^n
\E^\w_{U}(f^{(i)},f^{(i)}):=\frac{1}{2}\sum_{i=1}^n\sum_{x,y\in U: x\neq y}(f^{(i)}(x)-f^{(i)}(y))^2\frac{w_{x,y}(\w)}{|x-y|^{d+\alpha}}.
$$ It is not difficult to verify that for every finite subset $U\subset \Z^d$ and $f:U \to \R^n$,
\begin{equation}\label{e1-3a}
-\int_{U}\left\langle \LL_U^\w f(x), f(x)\right\rangle\mu(dx)=-\mathscr{E}_U^{\w}(f,f),
\end{equation}
where $\left\langle\cdot, \cdot \right\rangle$ is the inner product on $\R^d$.
In the rest of this paper, for simplicity of notation sometime we will omit the parameter
$\w$
when there is no danger of confusion.

\section{Preliminaries}\label{section2}

In this section,
we
establish two Poincar\'e-type inequalities,
a local Poincar\'e inequality and a  multi-scale Poincar\'e inequality,
for the Dirichlet form $(\E^{\w}, \F^{\w})$ under {\bf Assumption (H1)}, where the conductance may be degenerate.
First, we need a lemma.
For any $x_1,x_2,y\in \Z^d$ and $\delta, R>0$, let
$$B_{R,\delta}^{x_1,x_2}(y):
 =\{z\in B_R(y)\setminus \{x_1, x_2\} : w_{x_1,z}>\delta \hbox{ and } w_{x_2,z}>\delta \}.$$

\begin{lemma}\label{l2-0}Under {\bf Assumption (H1)}, there exist constants $\kappa,\delta_0\in (0,1]$  such that, for every $\theta>1/d$, there is a random variable $R_0(\w)\ge1$ so that for all $R>R_0(\w)$,
\begin{equation}\label{l2-0-1}
\mu(B_{r,\delta_0}^{x_1,x_2}(y))\ge \kappa\mu(B_r(y))
\quad \hbox{for } x_1,x_2,y\in B_{R}\hbox{ with } x_1\neq x_2  \hbox{ and } r\in [  \log^\theta R,    R ] .
\end{equation}
\end{lemma}

\begin{proof}
By {\bf Assumption (H1)}(ii),
there are
constants $\delta_0\in (0,1)$ and $p_0\in (0,1]$ such that
$\Pp(w_{x,y}>\delta_0)\ge p_0$
for all $x\not= y\in \Z^d$.
Fix $x_1\neq x_2\in \Z^d$, and define  $\eta_z^{x_1,x_2}(\w):=\I_{\{w_{x_1,z}>\delta_0,w_{x_2,z}>\delta_0\}}(\w)$ for any $z\in \Z^d \setminus\{x_1, x_2\}$.
According to {\bf Assumption (H1)}, $\{\eta_z^{x_1,x_2}:z\in \Z^d\setminus\{ x_1,  x_2\}\}$ is a sequence of i.i.d.\ and $\{0,1\}$-valued random variables with $\Pp\left(\eta_z^{x_1,x_2}=1\right)
\ge p_0^2$. Hence, by choosing $\kappa\in (0,p_0^2/2)$,
we get that for all $r>0$,
\begin{align*}
&\Pp\bigg(\sum_{z\in  B_r(y)
\setminus\{ x_1,  x_2\} }
\I_{\{w_{x_1,z}>\delta_0,w_{x_2,z}>\delta_0\}}\le \kappa \mu(B_r(y))\bigg)\\
&\le
\Pp\bigg(\frac{\sum_{z\in  B_r(y)
\setminus\{ x_1,  x_2\} }
\left(\eta_z^{x_1,x_2}-\Ee[\eta_z^{x_1,x_2}]\right)}{\mu(B_r(y))}
\le 2\kappa-p_0^2\bigg)
\le c_1e^{-c_2r^{d}},
\end{align*}
where
the last inequality is due to
Cramer-Chernoff's theorem (see \cite[Theorem 1.3, Remark 1.4 and Example 1.5(1) on p.\ 2-3]{BDR}). This implies that,
for every $\theta>1/d$
\begin{equation*}
\sum_{R=1}^\infty \sum_{x_1,x_2,y\in B_{R}\,\,{\rm with }\,\, x_1\neq x_2}\sum_{r=\log^\theta R}^R\Pp\bigg(\sum_{z\in B_r(y)
\setminus\{ x_1,  x_2\}
}\I_{\{w_{x_1,z}>\delta_0,w_{x_2,z}>\delta_0\}}\le \kappa\mu(B_r(y))\bigg)<\infty,
\end{equation*}
which along with Borel-Cantelli's lemma
yields \eqref{l2-0-1}.
\end{proof}

\begin{proposition}\label{l2-1}{\bf(Local Poincar\'e inequality)} Under {\bf Assumption (H1)}, there exists a constant $C_1>0$ such that, for any $\theta>1/d$, there is a random variable $R_0(\w)\ge1$ so that
for all $f:\Z^d \to \R$, $R>R_0(\w)$, $y \in B_R$ and $r \in [ \log^\theta R,  R]$,
\begin{equation}\label{l2-1-1}
\oint_{B_r(y)}f^2(x)\,\mu(dx)-\bigg(\oint_{B_r(y)}f(x)\,\mu(dx)\bigg)^2\le C_1r^{\alpha-d}\mathscr{E}_{B_r(y)}^{\w}(f,f).
\end{equation}
\end{proposition}

\begin{proof}
Note that, by \eqref{l2-0-1} there are a constant $\delta_0\in (0,1)$ and a random variable $R_0(\w)\ge 1$ such that
for all $R>R_0(\w)$, $x,x',y\in B_R$ with $x\neq x'$ and
$r\in [ \log^\theta R,  R]$,
we can find
$v(x,x')\in B_{r,\delta_0}^{x,x'}(y)$,  i.e.,
$v(x,x') \in B_r(y)\setminus \{x, x'\}$ with  $ w_{x, v(x,x')}>\delta_0$ and $w_{x',v(x,x')}>\delta_0$.
Thus, for all $f:\Z^d\to \R$, $R>R_0(\w)$, $y\in B_R$ and
$r\in [\log^\theta R,  R]$,
\begin{align*}
& \int_{B_r(y)}\Big(f(x)-
\oint_{B_r(y)}f\,d\mu
\Big)^2\,\mu(dx)\\
&=
\mu(B_r(y))^{-2}\sum_{x\in B_r(y)}\bigg(\sum_{x'\in B_r(y)
\setminus \{x\}}
(f(x)-f(x'))\bigg)^2\\
&\le 2\mu(B_r(y))^{-2}
\sum_{x\in B_r(y)}\bigg[\bigg(\sum_{x'\in B_r(y)\setminus\{ x\} }(f(x)-f(v(x,x')))\bigg)^2+\bigg(\sum_{x'\in B_r(y)\setminus\{ x\} }
(f(x')-f(v(x,x')))\bigg)^2\bigg]\\
&=:I_1^{y,r}+I_2^{y,r}.
\end{align*}
By the Cauchy-Schwartz inequality, it holds that
\begin{align}
I_1^{y,r}&\le c_1r^{-2d}\sum_{x\in B_r(y)}\bigg(
\sum_{x'\in B_r(y)\setminus\{ x\} }  (f(x)-f(v(x,x')) )^2\frac{w_{x,v(x,x')}}{|x-v(x,x')|^{d+\alpha}}\bigg)\nonumber\\
&\qquad\qquad\qquad\quad\times
\bigg(\sum_{x'\in B_r(y)\setminus\{ x\} } w_{x,v(x,x')}^{-1}|x-v(x,x')|^{d+\alpha}\bigg)\nonumber\\
&\le
c_2r^{\alpha}\sum_{x\in B_r(y)}\sum_{x'\in B_r(y)\setminus\{ x\} } (f(x)-f(v(x,x')))^2\frac{w_{x,v(x,x')}}{|x-v(x,x')|^{d+\alpha}},
\label{eq:jobfoem}\end{align}
where in the last inequality we have used the fact that $w_{x,v(x,x')}^{-1}\le \delta_0^{-1}$ and
$|x-v(x,x')|\le
2\sqrt{d} r$,
thanks to  the definition of $B^{x,x'}_{r,\delta_0}(y)$.

For every $x\in B_r(y)$ and $\delta>0$, set $B_{r,\delta}^x(y):=
\{z\in B_r(y)\setminus \{x\}: w_{x,z}>\delta\}$.
To estimate
the second sum in \eqref{eq:jobfoem} for fixed $x\in \Z^d$,
we will choose different $v(x,x')$ as much as
possible for different $x'$.
The worst case is that $B_{r,\delta_0}^{x,x'}(y)$ is the
same set
for all $x'\in B_r(y)$. According to
\eqref{l2-0-1} again, we know that in this case one can choose $\{v(x,x'): x'\in B_r(y)\}$ such that, for all $R>R_0(\w)$ and
$r\in [
\log^\theta R,
R]$,
each $u\in B_{r,\delta_0}^{x}(y)$ will be taken for at most $N:=\frac{\mu(B_r(y))}{\mu(B_{r,\delta_0}^{x,x'}(y))}\le \kappa^{-1}$ times for all different $x'\in B_r(y)$. By this observation, we can get \begin{align*}
\sum_{x\in B_r(y)}\sum_{x'\in B_r(y)\setminus \{x\} }\left(f(x)-f(v(x,x'))\right)^2\frac{w_{x,v(x,x')}}{|x-v(x,x')|^{d+\alpha}}
&\le \kappa^{-1}\sum_{x\in B_r(y)}\sum_{z\in B_{r,\delta_0}^x(y)}\left(f(x)-f(z)\right)^2\frac{w_{x,z}}{|x-z|^{d+\alpha}}\\
&\le \kappa^{-1}\sum_{x\in B_r(y)}\sum_{z\in B_r(y)}\left(f(x)-f(z)\right)^2\frac{w_{x,z}}{|x-z|^{d+\alpha}}.
\end{align*}

Combining both estimates above, we arrive at
\begin{align*}
I_1^{y,r} \le c_3r^\alpha\sum_{x\in B_r(y)}\sum_{z\in B_r(y)}\left(f(x)-f(z)\right)^2\frac{w_{x,z}}{|x-z|^{d+\alpha}}
=c_3r^\alpha \mathscr{E}_{B_r(y)}^{\w}(f,f).
\end{align*}

Similarly, we have  $I_2^{y,r} \leq c_4r^\alpha \mathscr{E}_{B_r(y)}^{\w}(f,f)$.
This establishes the desired assertion \eqref{l2-1-1}.
\end{proof}

Next, we establish
a
multi-scale Poincar\'e inequality, whose proof is partly motivated by that of \cite[Proposition 1.7]{AKM}.
For every
$m,n\in \N$ with $n<m$, define
$$
\Z_{m,n}^d:=\left\{z=(z_1,\cdots,z_d)\in B_{2^m}: z_i=k_i2^n\hbox{ for some odd}\ k_i\in
\Z
\hbox{ for all } 1\le i \le d\right\}.
$$
When $m=n$, we define  $\Z_{m,m}^d=\{0\}$.
Note that   $B_{2^m}=
\cup_{z\in \Z_{m,n}^d}
B_{2^n}(z)$
 and  $| \Z_{m,n}^d|=2^{d(m-n)}$.

\begin{proposition}\label{l2-2}{\bf(Multi-scale Poincar\'e inequality)}
Under {\bf Assumption (H1)}, there is a constant $C_2>0$ such that,
 for every $\theta>1/d$, there exists a random variable $m_0(\w)\ge 1$ such that for
 all integer $m>m_0(\w)$ and $n\in \N \cap [\theta \log_2 (m \log 2),  m]$,
and $f, g: B_{2^m}\to \R$,
\begin{equation}\label{l2-2-1}
\begin{split}
\sum_{x\in B_{2^m}}f(x)\bigg(g(x)-
\oint_{B_{2^m}}g \,d\mu
\bigg)\le&
\sum_{z\in \Z_{m,n}^d}\sum_{x\in B_{2^n}(z)}f(x)\bigg(g(x)-
\oint_{B_{2^n}(z)}g\,d\mu
\bigg)\\
&+C_2\mathscr{E}_{B_{2^m}}^\w(g,g)^{1/2}\sum_{k=n}^{m-1} 2^{k(d+\alpha)/2}\bigg(\sum_{y\in \Z_{m,k}^d}\bigg(\oint_{B_{2^k}(y)}f\,d\mu\bigg)^2\bigg)^{1/2}.
\end{split}
\end{equation}
\end{proposition}
\begin{proof} For any $1\le k<m$, $f:B_{2^m}\to \R$ and $g:B_{2^m}\to \R$,
\begin{equation}\label{l2-2-1a}
\begin{split}
&\sum_{z\in \Z_{m,k+1}^d}\int_{B_{2^{k+1}}(z)}f(x)\bigg(g(x)-\oint_{B_{2^{k+1}}(z)} g\,d\mu\bigg)\,\mu(dx)\\
&=\sum_{z\in \Z_{m,k+1}^d}\sum_{y\in \Z_{m,k}^d\cap B_{2^{k+1}}(z) }\int_{B_{2^{k}}(y)}f(x)\bigg(g(x)-\oint_{B_{2^{k+1}}(z)} g\,d\mu\bigg)\,\mu(dx)\\
&=\sum_{z\in \Z_{m,k+1}^d}\sum_{y\in \Z_{m,k}^d\cap B_{2^{k+1}}(z) }
\Bigg[\int_{B_{2^{k}}(y)}f(x)\left(g(x)-\oint_{B_{2^{k}}(y)} g\,d\mu\right)\,\mu(dx)\\
&\qquad \qquad\qquad\qquad\qquad\qquad+
\int_{B_{2^{k}}(y)}f(x)\left(\oint_{B_{2^{k}}(y)} g\,d\mu-\oint_{B_{2^{k+1}}(z)} g\,d\mu\right)\,\mu(dx)\Bigg]\\
&=\sum_{y\in \Z_{m,k}^d}\int_{B_{2^{k}}(y)}f(x)\left(g(x)-\oint_{B_{2^{k}}(y)} g\,d\mu\right)\,\mu(dx)\\
&\quad +\mu(B_{2^k})
\sum_{z\in \Z_{m,k+1}^d}\sum_{y\in \Z_{m,k}^d\cap B_{2^{k+1}}(z) }
\left(\oint_{B_{2^k}(y)}f\,d\mu\right)\left(\oint_{B_{2^{k}}(y)} g\,d\mu-\oint_{B_{2^{k+1}}(z)} g\,d\mu\right).
\end{split}
\end{equation}

According to the Cauchy-Schwartz inequality,
\begin{equation}\label{l2-2-1ab}
\begin{split}
&\mu(B_{2^k})\sum_{z\in \Z_{m,k+1}^d}\sum_{y\in \Z_{m,k}^d\cap B_{2^{k+1}}(z) }
\left(\oint_{B_{2^k}(y)}f\,d\mu\right)\left(\oint_{B_{2^{k}}(y)} g\,d\mu-\oint_{B_{2^{k+1}}(z)} g\,d\mu\right)\\
&\le c_12^{kd}\left(\sum_{z\in \Z_{m,k+1}^d}\sum_{y\in \Z_{m,k}^d\cap B_{2^{k+1}}(z) }
\left(\oint_{B_{2^{k}}(y)} g\,d\mu-\oint_{B_{2^{k+1}}(z)} g\,d\mu\right)^2\right)^{1/2}
\left(\sum_{y\in \Z_{m,k}^d}\left(\oint_{B_{2^k}(y)}f\,d\mu\right)^2\right)^{1/2} \\
 &\le c_12^{kd}\left(\sum_{z\in \Z_{m,k+1}^d}\sum_{y\in \Z_{m,k}^d\cap B_{2^{k+1}}(z) }
  \oint_{B_{2^{k}}(y)} \left(g(x)-\oint_{B_{2^{k+1}}(z)} g\,d\mu\right)^2 \,\mu(dx)    \right)^{1/2}
\left(\sum_{y\in \Z_{m,k}^d}\left(\oint_{B_{2^k}(y)}f\,d\mu\right)^2\right)^{1/2} .
\end{split}
\end{equation}

Note that, by (the proof of) Lemma \ref{l2-0}, there is $\delta_0\in (0,1)$ such that for all $\theta>1/d$ there exists a random variable $m_0(\w)\ge1$ so that for all
$m>m_0(\w)$,
$\theta \log_2 (m\log 2) \le k\le m$, $x\in B_{2^k}(y)$,
$x'\in B_{2^{k+1}}(z)
\setminus \{x\}$
(and with $y\in \Z_{m,k}^d\cap B_{2^{k+1}}(z) $ and $z\in \Z_{m,k+1}^d$),  we can choose
$v(x,x')\in B_{2^{k},\delta_0}^{x,x'}(y)$. Hence,
\begin{align*}
&  \oint_{B_{2^{k}}(y)} \left(g(x)-\oint_{B_{2^{k+1}}(z)} g\,d\mu\right)^2 \,\mu(dx)\\
&=\mu(B_{2^k}(y))^{-1}\mu(B_{2^{k+1}}(z))^{-2}\sum_{x\in B_{2^k}(y)}
\left(\sum_{x'\in B_{2^{k+1}}(z):x'\neq x}(g(x)-g(x'))\right)^2\\
&\le c_32^{-3kd}\sum_{x\in B_{2^k}(y)}\left(\left(
\sum_{x'\in B_{2^{k+1}}(z):x'\neq x}(g(x)-g(v(x,x')))\right)^2
+
\left(\sum_{x'\in B_{2^{k+1}}(z):x'\neq x}(g(x')-g(v(x,x')))\right)^2\right) .
\end{align*}

Further applying the Cauchy-Schwartz inequality we obtain
\begin{align*}
&\sum_{x\in B_{2^k}(y)}\left(
\sum_{x'\in B_{2^{k+1}}(z):x'\neq x}(g(x)-g(v(x,x')))\right)^2\\
&\le \sum_{x\in B_{2^k}(y)}\left(
\sum_{x'\in B_{2^{k+1}}(z):x'\neq x}(g(x)-g(v(x,x')))^2\frac{w_{x,v(x,x')}}{|x-v(x,x')|^{d+\alpha}}\right)\\
& \qquad \qquad\qquad \times
\left(\sum_{x'\in B_{2^{k+1}}(z):x'\neq x}w_{x,v(x,x')}^{-1}
\left|x-v(x,x')\right|^{d+\alpha}\right)\\
&\le c_42^{k(2d+\alpha)}\sum_{x\in B_{2^k}(y)}
\sum_{x'\in B_{2^{k+1}}(z):x'\neq x}(g(x)-g(v(x,x')))^2\frac{w_{x,v(x,x')}}{|x-v(x,x')|^{d+\alpha}},
\end{align*}
where the last inequality we used the fact that $w_{x,v(x,x')}^{-1}\le \delta_0^{-1}$, thanks to $v(x,x')\in B_{2^{k},\delta_0}^{x,x'}(y)$.

As explained in the proof of Proposition \ref{l2-1}, in order to estimate the item
above, for fixed $x\in B_{2^k}(y)$ we will choose different $v(x,x')\in B_{2^{k},\delta_0}^{x,x'}(y)$ as much as possible for
different $x'\in B_{2^{k+1}}(z)$.  Hence, using \eqref{l2-0-1} and following the same argument in the proof of Proposition \ref{l2-1}, we can find that
for every $m>m_0(\w)$ and
$k \in \Z \cap [ \theta \log_2 (m\log 2) , m]$,
  \begin{align*}
&\sum_{x\in B_{2^k}(y)}
\sum_{x'\in B_{2^{k+1}}(z)
\setminus \{x\}}
(g(x)-g(v(x,x')))^2\frac{w_{x,v(x,x')}}{|x-v(x,x')|^{d+\alpha}}\\
&\le
c_5 \sum_{x\in B_{2^k}(y)}
\sum_{x'\in B_{2^k}(y) \setminus \{x\} }
(g(x)-g(x'))^2\frac{w_{x,x'}}{|x-x'|^{d+\alpha}} \\
& \le c_5\sum_{x\in B_{2^k}(y)}
 \sum_{x'\in B_{2^{k+1}}(z) \setminus \{x\} }
(g(x)-g(x'))^2\frac{w_{x,x'}}{|x-x'|^{d+\alpha}},
\end{align*}
where in the last inequality we used the fact that $B_{2^k}(y)\subset B_{2^{k+1}}(z)$ for
every $y\in \Z_{m,k}^d\cap B_{2^{k+1}}(z) $ and $z\in \Z_{m,k+1}^d$.

In the same way,  we can obtain
\begin{align*}
\sum_{x\in B_{2^k}(y)}\left(
\sum_{x'\in B_{2^{k+1}}(z)\setminus \{x\} } (g(x')-g(v(x,x')))\right)^2\le
c_5 2^{k(2d+\alpha)}
\sum_{x\in B_{2^k}(y)}\sum_{x'\in B_{2^{k+1}}(z) \setminus \{x\} } (g(x)-g(x'))^2\frac{w_{x,x'}}{|x-x'|^{d+\alpha}}.
\end{align*}
Thus, for every $m>m_0(\w)$ and
$k \in \Z \cap [ \theta \log_2 (m\log 2) , m]$,
 \begin{align*}
\oint_{B_{2^{k}}(y)} \left(g(x)-\oint_{B_{2^{k+1}}(z)} g\,d\mu\right)^2 \,\mu(dx)\le
c_62^{-k(d-\alpha)}\int_{B_{2^k}(y)}\int_{B_{2^{k+1}}(z)}(g(x)-g(x'))^2\frac{w_{x,x'}}{|x-x'|^{d+\alpha}}\,\mu(dx)\,\mu(dx').
\end{align*}
Thus we have
\begin{align*}
 & \sum_{z\in \Z_{m,k+1}^d}\sum_{y\in \Z_{m,k}^d\cap B_{2^{k+1}}(z) }
\oint_{B_{2^{k}}(y)} \left(g(x)-\oint_{B_{2^{k+1}}(z)} g\,d\mu\right)^2 \,\mu(dx) \\
 &\le c_62^{-k(d-\alpha)}\sum_{z\in \Z_{m,k+1}^d}\sum_{y\in \Z_{m,k}^d\cap B_{2^{k+1}}(z) }
\int_{B_{2^k}(y)}\int_{B_{2^{k+1}}(z)}(g(x)-g(x'))^2\frac{w_{x,x'}}{|x-x'|^{d+\alpha}}\,\mu(dx)\,\mu(dx')\\
&\le c_62^{-k(d-\alpha)}
\int_{B_{2^m}}\int_{B_{2^m}}(g(x)-g(x'))^2\frac{w_{x,x'}}{|x-x'|^{d+\alpha}}\,\mu(dx)\,\mu(dx')=2c_62^{-k(d-\alpha)}
\mathscr{E}_{B_{2^m}}^\w(g,g).
\end{align*}

Combining this with \eqref{l2-2-1ab} and \eqref{l2-2-1a}, and taking the summation in \eqref{l2-2-1a} from $k=n$
(with $m>m_0(\w)$ and
$n\in \Z \cap [ \theta \log_2 (m\log 2 ) ,  m]  $)
to $k=m-1$,
we get  the desired inequality \eqref{l2-2-1}.
\end{proof}

\begin{remark}\label{rem-Poin}
When the sequence of i.i.d. non-negative random variables $\{w_{x,y}: (x,y)\in E\}$ is bounded away from zero; that is, there is a constant $C_0>0$ such that for all $(x,y)\in E$, $w_{x,y}\ge C_0$. Then, the local Poincar\'e inequality \eqref{l2-1-1} can be verified by directly applying the Cauchy-Schwartz inequality (see the proof of \cite[Theorem 3.1]{CK08}). In particular, in this case we can
show
 that the local Poincar\'e inequality \eqref{l2-1-1} holds for all $r\ge1$, and so the multi-scale Poincar\'e inequality \eqref{l2-2-1} holds for all $1\le n\le m$.
\end{remark}

\section{Energy Estimates of Solutions to Localized Poisson Equation} \label{section2-}

In this section, we are concerned with energy estimates (i.e., $H^{\alpha/2}$-bounds) for solutions to
localized Poisson equations associated with the operator $\sL_U^{\w}$ on compact sets $U$,  which
is a crucial ingredient  for the quantification of the corresponding stochastic homogenization. In particular, to drive local $H^{\alpha/2}$-bounds  we will make use of the multi-scale Poincar\'e inequality established in Proposition \ref{l2-2}.

Suppose that {\bf Assumption (H1)} holds.
First, we consider the case that
$\alpha\in (1,2)$. According to Proposition \ref{l2-1}, there is a random variable $m_0(\w)\ge 1$ such that for all
$m>m_0(\w)$, there exists a unique solution $\phi_m:B_{2^m}\to \R^d$ to the following (local Poisson) equation
\begin{equation}\label{e2-1}
\begin{cases}
 \sL^\w_{B_{2^m}} \phi_m (x)=-V(x)+\displaystyle
 \oint_{B_{2^m}} V\,d\mu,
 \quad  x\in B_{2^m},\\
 \sum_{x\in B_{2^m}}  \phi_m(x)=0,
\end{cases}
\end{equation}
where $\sL^\w_{B_{2^m}}$ is defined by \eqref{e1-1a} with $U=B_{2^m}$ and
\begin{equation}\label{e2-2}
V(x):=\sum_{z\in \Z^d} \frac{z}{|z|^{d+\alpha}}w_{x,x+z}.
\end{equation} Indeed, $ \sL^\w_{B_{2^m}}$ is the generator of the Dirichlet form $(\mathscr{E}_{B_{2^m}}^{\w}, L^2( B_{2^m};\mu ))$, which is associated with a reflected symmetric $\alpha$-stable-like process $X^{\w,B_{2^m}}$ on $B_{2^m}$.
It follows from  Proposition \ref{l2-1} that there is a random variable $m_0(\w)\ge 1$ such that for all
$m>m_0(\w)$, the process $X^{B_{2^m}}$ is exponentially ergodic. This yields the existence and the uniqueness of the solution to the equation \eqref{e2-1}.

\begin{proposition}\label{p2-1} Assume that {\bf Assumption (H1)} holds.
Let $\phi_m: B_{2^m}\mapsto \R^d$ be the $($unique$)$ solution
to \eqref{e2-1}.
Suppose that $\alpha\in (1,2)$
and $d>\alpha$.
Then, for any $\gamma>0$, there exist a constant $C_3>0$ and a random variable $m^*_1(\omega)\ge1$ such that for every $m>m^*_1(\w)$,
\begin{equation}\label{p2-1-2}
 \mathscr{E}_{B_{2^m}}^\w(\phi_m,\phi_m) \le C_3
m^{\frac{(1+\gamma)\alpha}{(2(d-\alpha))\wedge d}}2^{md}.
\end{equation}
\end{proposition}
\begin{proof} Let $\phi_m:=(\phi_m^{(1)},\phi_m^{(2)},\cdots,\phi_m^{(d)})$ and $V(x)=(V^{(1)}(x),V^{(2)}(x),\cdots,V^{(d)}(x))$. In particular,
$$\mathscr{E}_{B_{2^m}}^\w(\phi_m,\phi_m)=\sum_{i=1}^d \mathscr{E}_{B_{2^m}}^\w(\phi_m^{(i)},\phi_m^{(i)}).$$
According to
\eqref{e1-3a} and
\eqref{e2-1}, for $1\le i\le d$,
\begin{align*}
\mathscr{E}_{B_{2^m}}^\w(\phi_m^{(i)},\phi_m^{(i)})&=-
\int_{B_{2^m}}\sL^\w_{B_{2^m}}\phi_m^{(i)}(x)\cdot
\phi_m^{(i)}(x)\,\mu(dx)= \int_{B_{2^m}}\left(V^{(i)}(x)-\oint_{B_{2^m}}V^{(i)}\,d\mu\right)\cdot\phi_m^{(i)}(x)\,\mu(dx).
\end{align*}
Then, applying \eqref{l2-2-1}
with
$f=V^{(i)}-\displaystyle\oint_{B_{2^m}}V^{(i)}\,d\mu$ and $g=\phi_m^{(i)}$,
for any $\gamma>0$ there is $m_1(\w)\ge 1$ such that for all
$m>m_1(\w)$ and $[d^{-1}(1+\gamma)\log_2 m] \le n \le m$,
\begin{align*}
 \mathscr{E}_{B_{2^m}}^\w(\phi_m,\phi_m)
\le & \sum_{i=1}^d\sum_{z\in \Z_{m,n}^d}\int_{B_{2^n}(z)}
\left(V^{(i)}(x)-\oint_{B_{2^m}}V^{(i)}\,d\mu\right)\cdot\left(\phi_m^{(i)}(x)-\oint_{B_{2^n}(z)}\phi_m^{(i)}\,d\mu\right)\,\mu(dx)\\
&+c_1\sum_{i=1}^d\mathscr{E}_{B_{2^m}}^\w(\phi_m^{(i)},\phi_m^{(i)})^{1/2}\sum_{k=n}^{m-1} 2^{k(d+\alpha)/2}\left(\sum_{y\in \Z_{m,k}^d}\left(\oint_{B_{2^k}(y)}V^{(i)}\,d\mu
-\oint_{B_{2^m}}V^{(i)}\,d\mu\right)^2\right)^{1/2}\\
= &:J_{m,n,1}+J_{m,n,2}.
\end{align*}

Let $C_1>0$ be the constant in \eqref{l2-1-1}. Then, by Young's inequality,
\begin{equation}\label{p2-1-3}
\begin{split}
J_{m,n,1}&\le  2C_12^{n\alpha}\sum_{z\in \Z_{m,n}^d}\int_{B_{2^n}(z)}\left|V(x)-\oint_{B_{2^m}}V\,d\mu\right|^2\mu(dx) \\
&\quad+\frac{2^{-n  \alpha }}{8C_1}\sum_{z\in \Z_{m,n}^d}\int_{B_{2^n}(z)}\left|\phi_m(x)-\oint_{B_{2^n}(z)}\phi_m\,d\mu\right|^2\mu(dx)\\
&\le c_22^{md+n\alpha }+\frac{1}{8}
\sum_{i=1}^d
\sum_{z\in \Z_{m,n}^d}\int_{B_{2^n}(z)}
\int_{B_{2^n}(z)}(\phi_m^{(i)}(x)-\phi_m^{(i)}(y))^2\frac{w_{x,y}}{|x-y|^{d+\alpha}}\,\mu(dx)\,\mu(dy)\\
&\le c_22^{md+n\alpha }+
\frac{1}{4} \mathscr{E}_{B_{2^m}}^\w(\phi_m,\phi_m),
\end{split}
\end{equation}
where the second inequality follows from \eqref{l2-1-1} and the fact that $\sup_{x\in \Z^d}|V(x)|\le c_3$, and in the last inequality we have used the property that
$\sum_{z\in\Z_{m,n}^d} \mathscr{E}_{B_{2^n}}^\w(\phi_m^{(i)},\phi_m^{(i)})
\le \mathscr{E}_{B_{2^m}}^\w(\phi_m^{(i)},\phi_m^{(i)})$,
thanks to  $\sum_{z\in \Z_{m,n}^d} B_{2^n}(z) =B_{2^m}$.

On the other hand, for every $1\le k \le m$,
\begin{align*}
\sum_{y\in \Z^d_{m,k}}\left|\oint_{B_{2^k}(y)}V\,d\mu\right|^2&=2^{-2(kd+1)}\sum_{i=1}^d\sum_{y\in \Z^d_{m,k}}\sum_{z_1,z_2\in B_{2^k}(y)}V^{(i)}(z_1)V^{(i)}(z_2)\\
&=2^{-2(kd+1)}\sum_{i=1}^d\sum_{y\in \Z^d_{m,k}}\sum_{z_1,z_2\in B_{2^k}(y)}
\sum_{z_3,z_4\in \Z^d}\frac{(z_1-z_3)^{(i)}}{|z_1-z_3|^{d+\alpha}}
\frac{(z_2-z_4)^{(i)}}{|z_2-z_4|^{d+\alpha}}\xi_{z_1,z_3}\xi_{z_2,z_4},
\end{align*}
where $\xi_{x,y}:=w_{x,y}-1$ for all $x,y\in \Z^d$, and  in the second equality we used the fact
that
$$V(x)=\sum_{z\in \Z^d}\frac{z}{|z|^{d+\alpha}}w_{x,x+z}=\sum_{z\in \Z^d}\frac{z}{|z|^{d+\alpha}}
\left(w_{x,x+z}-1\right).$$
Thus,
\begin{align*}
&\Ee\left[\left(\sum_{y\in \Z^d_{m,k}}\left|\oint_{B_{2^k}(y)}V\,d\mu\right|^2\right)^2\right]\\
&=2^{-4(kd+1)}\sum_{i,i'=1}^d\sum_{y,y'\in \Z^d_{m,k}}\sum_{z_1,z_2\in B_{2^k}(y)}\sum_{z_1',z_2'\in B_{2^k}(y')}
\sum_{z_3,z_4,z_3',z_4'\in \Z^d}
\frac{(z_1-z_3)^{(i)}(z_1'-z_3')^{(i')}}{(|z_1-z_3||z_1'-z_3'|)^{d+\alpha}}\\
&\hskip 3truein \times
\frac{(z_2-z_4)^{(i)}(z_2'-z_4')^{(i')}}{(|z_2-z_4||z_2'-z_4'|)^{d+\alpha}}
\Ee\left[\xi_{z_1,z_3}\xi_{z_2,z_4}\xi_{z_1',z_3'}
\xi_{z_2',z_4'}
\right].
\end{align*}

Note that $\Ee[\xi_{x,y}]=0$. According to {\bf Assumption (H1)}, it is easy to see that
\begin{equation*}
\Ee\left[\xi_{z_1,z_3}\xi_{z_2,z_4}\xi_{z_1',z_3'}\xi_{z_2',z_4'}\right]\neq 0
\end{equation*}
only if
in one of the following three cases:
\begin{itemize}
\item [(a)] $(z_1,z_3)=(z_2,z_4)$, $(z_1',z_3')=(z_2',z_4')$;
\medskip
\item [(b)] $(z_1,z_3)=(z_1',z_3')$, $(z_2,z_4)=(z_2',z_4')$;
\medskip
\item [(c)] $(z_1,z_3)=(z_2',z_4')$, $(z_2,z_4)=(z_1',z_3')$.
\end{itemize}
Hence,
\begin{align*}
& \Ee\left[\left(\sum_{y\in \Z^d_{m,k}}\left|\oint_{B_{2^k}(y)}V\,d\mu\right|^2\right)^2\right]\le
c_4\sum_{l=1}^2 J_{l,k,m},
\end{align*}
where
\begin{align*}
J_{1,k,m}:=& 2^{-4(kd+1)}\sum_{i,i'=1}^d\sum_{y,y'\in \Z^d_{m,k}}\sum_{(z_1,z_3)\in E: z_1\in B_{2^k}(y)}
\sum_{(z_1',z_3')\in E: z_1'\in B_{2^k}(y')}\frac{|(z_1-z_3)^{(i)}|^2}{|z_1-z_3|^{2(d+\alpha)}}\\
& \hskip 3truein
\times  \frac{|(z_1'-z_3')^{(i')}|^2}{|z_1'-z_3'|^{2(d+\alpha)}}
\Ee\left[|\xi_{z_1,z_3}|^2 |\xi_{z_1',z_3'}|^2\right]
\end{align*} and
\begin{align*}
J_{2,k,m}:=& 2^{-4(kd+1)}\sum_{i,i'=1}^d\sum_{y\in \Z^d_{m,k}}\sum_{(z_1,z_3)\in E: z_1\in B_{2^k}(y)}
\sum_{(z_2,z_4)\in E: z_2\in B_{2^k}(y)}\\
&\qquad\quad \times \frac{|(z_1-z_3)^{(i)}||(z_1-z_3)^{(i')}|}{ |z_1-z_3| ^{2(d+\alpha)}} \frac{|(z_2-z_4)^{(i)}||(z_2-z_4)^{(i')}|}{ |z_2-z_4| ^{2(d+\alpha)}}
\Ee\left[|\xi_{z_1,z_3}|^2|\xi_{z_2,z_4}|^2\right].
\end{align*}
Since $|\xi_{x,y}|\le c_5$, we can deduce that
for every $1\le k \le m$ and $y\in \Z_{m,k}^d$,
\begin{align*}
J_{1,k,m}&\le c_62^{-4(kd+1)}
\sum_{y,y'\in \Z^d_{m,k}}\sum_{(z_1,z_3)\in E: z_1\in B_{2^k}(y)}
\sum_{(z_1',z_3')\in E: z_1'\in B_{2^k}(y')}\frac{1}{(|z_1-z_3||z_1'-z_3'|)^{2(d+\alpha-1)}}\\
&\le c_72^{-2kd+2(m-k)d}
\end{align*} and
\begin{align*}
J_{2,k,m}
&\le c_62^{-4(kd+1)}\sum_{y\in \Z^d_{m,k}}
\sum_{(z_1,z_3)\in E: z_1\in B_{2^k}(y)}
\sum_{(z_2,z_4)\in E: z_2\in B_{2^k}(y)}\frac{1}{(|z_1-z_3||z_2-z_4|)^{2(d+\alpha-1)}}\\
&\le c_72^{-2kd+(m-k)d}.
\end{align*}
Putting all the estimates yields that
\begin{align*}
&\Ee\left[\left(\sum_{y\in \Z_{m,k}^d}\left|\oint_{B_{2^k}(y)}V\,d\mu\right|^2\right)^2\right]\le
c_82^{-2kd+2(m-k)d}.
\end{align*}
By the Markov inequality, we can derive that for every $\theta\in ({\alpha}/{d},1)$ and $\gamma>0$,
\begin{align*}
&\quad\Pp\left(\bigcup_{m\ge 1}\bigcup_{[\frac{(1+\gamma)\log_2 m}{2(1-\theta)d}]\le k \le m}\left\{\left|
\sum_{y\in \Z_{m,k}^d}\left|\oint_{B_{2^k}(y)}V\,d\mu\right|^2\right|>2^{(m-k)d-\theta k d}\right\}\right)\\
&\le \sum_{m=1}^\infty\sum_{k=[\frac{(1+\gamma)\log_2 m}{2(1-\theta)d}]}^m2^{2\theta k d-2(m-k)d}
\Ee\left[\left(\sum_{y\in \Z_{m,k}^d}\left|\oint_{B_{2^k}(y)}V\,d\mu\right|^2\right)^2\right]\\
&\le c_8\sum_{m=1}^\infty \sum_{k=[\frac{(1+\gamma)\log_2 m}{2(1-\theta)d}]}^m 2^{-2k(1-\theta)d}\le c_9\sum_{m=1}^\infty m^{-(1+\gamma)}<\infty.
\end{align*}
This along with Borel-Cantelli's lemma yields that there is $m_2(\w)\ge 1$ such that for all $m\ge m_2(\w)$ and $[\frac{(1+\gamma)\log_2 m}{2(1-\theta)d}]\le
k\le m$,
\begin{equation}\label{p2-1-4}
\sum_{z\in \Z_{m,k}^d}\left|\oint_{B_{2^k}(z)}V\,d\mu\right|^2\le 2^{(m-k)d-{\theta kd} },
\end{equation}
which in turn implies that for all
$m\ge n\ge \left[\frac{(1+\gamma)\log_2 m}{2(1-\theta)d}\right]$,
\begin{align*}
J_{m,n,2}&\le c_{10}  \mathscr{E}_{B_{2^m}}^\w(\phi_m,\phi_m) ^{1/2}\left(\sum_{k=n}^{m-1} 2^{k(d+\alpha)/2+(m-k)d/2-
kd\theta/2
} \right)\\
&=c_{10} 2^{md/2} \mathscr{E}_{B_{2^m}}^\w(\phi_m,\phi_m) ^{1/2}\left(\sum_{k=n}^{m-1}
2^{-k(\theta d-\alpha)/2}
\right)\\
& \le c_{11}2^{md/2} \mathscr{E}_{B_{2^m}}^\w(\phi_m,\phi_m) ^{1/2}\le c_{12}2^{md}+\frac{1}{4}\mathscr{E}_{B_{2^m}}^\w(\phi_m,\phi_m),
\end{align*}
where the second inequality is due to
$\theta\in (\frac{\alpha}{d},1)$.

Combining
all the estimates above, taking
$n=\max\left(\left[\frac{(1+\gamma)\log_2 m}{2(1-\theta)d}\right],\left[\frac{(1+\gamma)\log_2 m}{d}\right]\right)$ and letting $\theta$ close to $\frac{\alpha}{d}$, we can  get
the desired assertion \eqref{p2-1-2}. \end{proof}

Next, we consider the case that $\alpha\in (0,1]$. For any $m\in \N$, let
$\widetilde \phi_m:
B_{2^m}\to \R^d$ be the (unique) solution to the following equation
\begin{equation}\label{e2-1-2--}
\begin{cases}
 \sL^\w_{B_{2^m}} \widetilde \phi_{m}(x)=-V_m(x)+\displaystyle \oint_{B_{2^m}} V_m\,d\mu,\quad  x\in B_{2^m},\\
 \sum_{x\in V_m}  \widetilde \phi_{m} (x)=0,
\end{cases}
\end{equation}
where
\begin{equation}\label{e2-2-2--}
V_m(x):=\sum_{z\in \Z^d: |z|\le 2^ m} \frac{z}{|z|^{d+\alpha}}w_{x,x+z}.
\end{equation}

\begin{proposition}\label{p2-1-}
Suppose that {\bf Assumption (H1)} holds,  $\alpha\in (0,1]$
and that $d>\alpha$.
Then, for any $\gamma>0$, there exist a constant $C_4>0$ and a random variable $m^*_2(\omega)\ge1$ such that for every $m>m^*_2(\w)$,
\begin{equation}\label{p2-1-2-}
 \mathscr{E}_{B_{2^m}}^\w(\widetilde \phi_m,\widetilde \phi_m)
\le C_4 \begin{cases}
m^{2+\frac{1+\gamma}{(2(d-1))\wedge d}}2^{md}  \qquad
&\hbox{when }  \alpha=1,\\
 m^{\frac{(1+\gamma)\alpha}{(2(d-\alpha))\wedge d}}
 2^{m(d+2(1-\alpha))} &\hbox{when } \alpha\in (0,1).
\end{cases}
\end{equation}
\end{proposition}
\begin{proof} It is obvious that there exists a constant $c_1>0$ such that, $\sup_{x\in \Z^d}|V_m(x)|\le c_1 m $ for all $m\ge1$ when $\alpha=1$, and $\sup_{x\in \Z^d}|V_m(x)|\le c_1 2^{m(1-\alpha)} $ for all $m\ge1$ when $\alpha\in (0,1)$.
Using
this fact and following the proof of Proposition \ref{p2-1}, we can obtain the desired assertion. \end{proof}

\begin{remark}\label{e:order} As mentioned in Remark \ref{rem-Poin}(1), when the sequence of i.i.d. non-negative random variables $\{w_{x,y}: (x,y)\in E\}$ is bounded away from zero,
the multi-scale Poincar\'e inequality \eqref{l2-2-1} holds for all $1\le n\le m$. Consequently, according to the arguments above, the assertion \eqref{p2-1-2} in Proposition \ref{p2-1} can be
improved as
$$\mathscr{E}_{B_{2^m}}^\w(\phi_m,\phi_m) \le C_3
m^{\frac{(1+\gamma)\alpha}{2(d-\alpha)}}2^{md}.$$
Similarly, \eqref{p2-1-2-} in the statement in Proposition \ref{p2-1-} can be
strengthened to
$$
 \mathscr{E}_{B_{2^m}}^\w(\widetilde \phi_m,\widetilde \phi_m)
\le C_4 \begin{cases}
m^{2+\frac{1+\gamma}{2(d-1)}}2^{md}  \qquad
&\hbox{when }  \alpha=1,\\
 m^{\frac{(1+\gamma)\alpha}{2(d-\alpha)}}
 2^{m(d+2(1-\alpha))}  &\hbox{when } \alpha\in (0,1).
\end{cases}$$
\end{remark}

 \section{Convergence of scaled operators  $\{\sL^{k,\w}\}_{k\ge1}$}\label{section3}
In order to obtain quantitative results for the stochastic homogenization, we shall consider the corresponding  quantitative difference between the scaled operators  $\{\sL^{k,\w}\}_{k\ge1}$   and the limit operator $\bar \sL$.
Recall the definition of $\bar \sL$ in \eqref{e:1.2}. It is then
natural to define the following operator defined on ${\mathcal B}_b(k^{-1}\Z^d)$ (the set of bounded measurable functions on $k^{-1}\Z^d$), which can be viewed as a discrete approximation to the limit operator $\bar \sL$:
\begin{equation}\label{e3-1a}
\begin{split}
  \bar \sL^{k}f(x):= &k^{-d}\sum_{z\in k^{-1}\Z^d}\left(f(x+z)-f(x)\right)\frac{1}{|z|^{d+\alpha}}\\
  =&\begin{cases}
  \displaystyle k^{-d}\sum_{z\in k^{-1}\Z^d}\left(f(x+z)-f(x)-\nabla f(x)\cdot z \I_{\{|z|\le 1\}}\right)\frac{1}{|z|^{d+\alpha}}   \qquad &\hbox{when } \alpha\in (0,1],\\
  \displaystyle k^{-d}\sum_{z\in k^{-1}\Z^d}\left(f(x+z)-f(x)-\nabla f(x)\cdot z \right)\frac{1}{|z|^{d+\alpha  }}
	&\hbox{when } \alpha\in (1,2).
  \end{cases}
  \end{split}
  \end{equation}

\medskip

 \begin{proposition}\label{l3-1-0}
For any
$f\in C_c^2 (\R^d)$,
there is a constant $C_5>0$ such that for all $k\ge 1$,
\begin{equation}\label{l3-1-2}
\int_{k^{-1}\Z^d}  |\bar \sL^{k}f(x)-\bar \sL f(x) |^2\,
\mu^{k}(dx)\le  C_5 \begin{cases} k^{-2}  \qquad &\hbox{when }  \alpha\in (0,1),\\
  k^{-2}(\log k)^2    &\hbox{when } \alpha=1,\\
   k^{-2(2-\alpha)}   &\hbox{when } \alpha\in (1,2).
	\end{cases}
\end{equation}
\end{proposition}

\begin{proof}
Without loss of generality, we
may and do
assume that $\text{supp}f\subset B_1$.
We consider the two cases separately.

{\bf Case 1:} $\alpha\in (1,2)$.  Set $$H(x,z):=\frac{f(x+z)-f(x)-
\langle \nabla f(x), z\rangle
}{|z|^{d+\alpha}}.$$  We get from the mean value theorem that for every $k\ge 4\sqrt{d},$ $x\in B_3\cap k^{-1}\Z^d$, $z\in k^{-1}\Z^d$ and $y\in  \prod_{1\le i\le d}(z_i,z_i+k^{-1}]$,
\begin{align*}
&|H(x,z)-H(x,y)|\\
&\le \frac{1}{2}\|\nabla^2 f\|_\infty\left(|z|^{2-d-\alpha}+|y|^{2-d-\alpha}\right)\I_{\{|z|\le 2\sqrt{d}/k\}}\\
&\qquad+\Bigg(\left|f(x+z)-f(x)-
\langle \nabla f(x), z\rangle
\right|\left|\frac{1}{|y|^{d+\alpha}}-\frac{1}{|z|^{d+\alpha}}\right|\\
&\qquad\quad +\left|\left(f(x+z)-f(x)-
\langle \nabla f(x), z\rangle
\right)-\left(f(x+y)-f(x)-
\langle \nabla f(x), y\rangle
\right)\right|
\frac{1}{|y|^{d+\alpha}}\Bigg)\I_{\{|z|>2\sqrt{d}/k\}}\\
&\le \frac{1}{2}\|\nabla^2 f\|_\infty\left(|z|^{2-d-\alpha}+|y|^{2-d-\alpha}\right)\I_{\{|z|\le 2\sqrt{d}/k\}}\\
&\quad+\Bigg[\frac{1}{2}\|\nabla^2 f\|_\infty|z|^2
\left|\frac{1}{|y|^{d+\alpha}}-\frac{1}{|z|^{d+\alpha}}\right|+\left(\int_0^1 |\nabla f(x+y+s(z-y))-\nabla f(x)|\,ds\right)\cdot \left(\frac{|y-z|}{|y|^{d+\alpha}}\right)\Bigg]\I_{\{|z|>2\sqrt{d}k^{-1}\}}\\
&\le c_7(|z|^{2-d-\alpha}+|y|^{2-d-\alpha})\I_{\{|z|\le 2\sqrt{d}/k\}}+
c_7k^{-1}|y|^{1-d-\alpha}\I_{\{|z|>2\sqrt{d}/k\}}.
\end{align*}
On the other hand, we can get from the fact ${\rm supp} [f]\subset B(0,1)$ that for every $x\in B_3^c\cap k^{-1}\Z^d$, $z\in k^{-1}\Z^d$ and $y\in  \prod_{1\le i\le d}(z_i,z_i+k^{-1}]$,
\begin{align*}
&\quad |H(x,z)-H(x,y)|\le c_8k^{-1}|z|^{-d-\alpha}\I_{\{z+x\in B_2\}}(x).
\end{align*}
Therefore, for every $x\in  k^{-1}\Z^d$ and $k\ge 4\sqrt{d}$,
\begin{align*}
&|\bar \sL f(x)-\bar\sL^{k}f(x)|\\
&=
\left|\sum_{z\in k^{-1}\Z^d}\int_{ \prod_{1\le i\le d}(z_i,z_i+k^{-1}]}\left(H(x,z)-H(x,y)\right)dy\right|\\
&\le c_9\Bigg(\int_{\{|y|\le 3\sqrt{d}/k\}}|y|^{2-d-\alpha}\,dy+
k^{-d}\sum_{z\in k^{-1}\Z^d:|z|\le 2\sqrt{d}/k}|z|^{2-d-\alpha}+
\int_{\{|y|>\sqrt{d}/k\}}
k^{-1}|y|^{1-d-\alpha}\,dy
\Bigg)\I_{B_3}(x)\\
&\quad +c_{10} \left(k^{-d} k^{-1}\sum_{z\in k^{-1}\Z^d: x+z\in B_2} |z|^{-d-\alpha}\right)\I_{B_3^c}(x)\\
&\le c_{11}\left(k^{-(2-\alpha)}\I_{B_3}(x)+k^{-1}|x|^{-d-\alpha}\I_{B_3^c}(x)\right),
\end{align*}
which yields that \eqref{l3-1-2} holds for all $k\ge 4\sqrt{d}$.
Note that there exists a constant $c_{12}>0$ such that for all $k\ge 1$ and
$f\in C_c^2(\R^d)$
with ${\rm supp} [f]\subset B(0,1)$,
\begin{equation}\label{e:remark1}\int_{k^{-1}\Z^d}  |\bar \sL^{k}f(x)-\bar \sL f(x) |^2\,
\mu^{k}(dx)\le 2\int_{k^{-1}\Z^d} |\bar \sL^{k}f(x)|^2\,
\mu^{k}(dx)+ 2\int_{k^{-1}\Z^d} |\bar \sL f(x)|^2\,
\mu^{k}(dx)\le c_{12}.
\end{equation}
The desired assertion \eqref{l3-1-2} (for all $k\ge 1$) now follows immediately.

\medskip

{\bf Case 2:} $\alpha\in (0,1]$. In this case, we define $$H(x,z):=\frac{f(x+z)-f(x)-
\langle \nabla f(x), z\I_{\{|z|\le 1\}}\rangle
}{|z|^{d+\alpha}}.$$
Applying the same arguments as in {\bf Case 1} and using the mean value theorem,
we find that for every $k\ge 4\sqrt{d}$, $x\in B_3\cap k^{-1}\Z^d$, $z\in k^{-1}\Z^d$ and $y\in  \prod_{1\le i\le d}(z_i,z_i+k^{-1}] $,
\begin{align*}
&|H(x,z)-H(x,y)|\\
&\le \frac{1}{2}\|\nabla f\|_\infty\left(|z|^{2-\alpha-d}+|y|^{2-\alpha-d}\right)\I_{\{|z|\le 2\sqrt{d}/k\}}\\
&\quad+\Bigg(\left|f(x+z)-f(x)-
\langle \nabla f(x), z \I_{\{|z|\le 1- 1/k \}}\rangle
\right|\left|\frac{1}{|y|^{d+\alpha}}-\frac{1}{|z|^{d+\alpha}}\right|\\
&\qquad\quad +\left|\left(f(x+z)-f(x)-
\langle \nabla f(x),z\I_{\{|z|\le 1- 1/k\}}\rangle
\right)-\left(f(x+y)-f(x)-
\langle \nabla f(x), y\I_{\{|z|\le 1- 1/k\}}\rangle
\right)\right|
\frac{1}{|y|^{d+\alpha}}\\
&\qquad\quad +\frac{
|\langle \nabla f(x), z\rangle|
}{|z|^{d+\alpha}}\I_{\{1-1/k\le |z|\le1\}}+
\frac{
|\langle \nabla f(x), y\rangle|
}{|y|^{d+\alpha}}\I_{\{1-(1+\sqrt{d})/k\le |y|\le 1+(1+\sqrt{d})/k\}}\Bigg)\I_{\{|z|>2\sqrt{d}/k\}}\\
&\le c_7\left(|z|^{2-\alpha-d}+|y|^{2-\alpha-d}\right)\I_{\{|z|\le 2\sqrt{d}/k\}}+
c_7k^{-1}|y|^{-d-\alpha+1}
\I_{\{2\sqrt{d}/k<|z|\le 1- 1/k\}}\\
&\quad+ c_7 \left(k^{-1}|y|^{-d-1-\alpha}+k^{-1}|y|^{-d-\alpha}\right)\I_{\{|z|> 1- 1/k\}}+c_7\left(\I_{\{1-1/k\le |z|\le1\}}+\I_{\{1-(1+\sqrt{d})/k\le |y|\le 1+(1+\sqrt{d})/k\}}\right).
\end{align*}
On the other hand, it is easy to see that  for every $x\in B_3^c\cap 2^{-m}\Z^d$, $z\in k^{-1}\Z^d$ and $y\in  \prod_{1\le i\le d}(z_i,z_i+k^{-1}]$,
\begin{align*}
&\quad |H(x,z)-H(x,y)|\le c_8k^{-1}|z|^{-d-\alpha}\I_{\{z+x\in B_2\}}(x).
\end{align*}
Hence, when $\alpha=1$, for every $x\in  k^{-1}\Z^d$ and $k\ge 4\sqrt{d}$,
\begin{align*}
& |\bar \sL f(x)-\bar \sL^{k}f(x) |\\
&=
\left|\sum_{z\in k^{-1}\Z^d}\int_{ \prod_{1\le i\le d}(z_i,z_i+k^{-1}]}\left(H(x,z)-H(x,y)\right)dy\right|\\
&\le c_9\Bigg(\int_{\{|y|\le 3\sqrt{d}/k\}}|y|^{1-d}\,dy+
k^{-d}\sum_{z\in k^{-1}\Z^d:|z|\le 2\sqrt{d}/k}|z|^{1-d}\\
&\qquad+
\int_{\{\sqrt{d}k^{-1}<|y|\le 2\}}
k^{-1}|y|^{-d}\,dy
+ \int_{\{|y|>1-2\sqrt{d}/k\}} k^{-1}|y|^{-d-1}\, dy+k^{-1}\Bigg)\I_{B_3}(x)\\
&\quad+c_{10}\left(k^{-d} k^{-1}\sum_{z\in k^{-1}\Z^d: x+z\in B_2} |z|^{-d-1}\right)\I_{B_3^c}(x)\\
&\le c_{11} (k^{-1} \log k\,\I_{B_3}(x)+k^{-1}|x|^{-d-1}\I_{B_3^c}(x) ).
\end{align*}
This gives us the second assertion when $\alpha=1$ and $k\ge 4\sqrt{d}$. Similarly, we can prove the second assertion when $\alpha\in (0,1)$ and $k\ge 4\sqrt{d}$.
This,  thanks to \eqref{e:remark1},  gives \eqref{l3-1-2} when $\alpha\in (0,1]$.
\end{proof}

In the following we will
consider
the differences among $\{\bar \sL^{k}\}_{k\ge1}$ and the scaled operators  $\{\sL^{k,\w}\}_{k\ge1}$,  as well as their variants. For simplicity, we will omit
the random environment
$\w \in \Omega$ from the notation.

Motivated by the definitions \eqref{e:operator} and \eqref{e3-1a}, we introduce the following
variant
of the scaled operators  $\{\sL^{k,\w}\}_{k\ge1}$.
For every $f\in C_b (\R^d)$,
$k\in \N_+$ and  $x\in k^{-1}\Z^d$,
we
define a generator with random coefficient
\begin{equation}\label{e:operator0}
 \widehat \sL^{k}f(x):=\begin{cases}
\displaystyle  k^{-d}\sum_{z\in k^{-1}\Z^d}\left(f(x+z)-f(x)-
\langle \nabla f(x), z \I_{\{|z|\le 1\}}\rangle
\right)\frac{w_{kx,k(x+z)}}{|z|^{d+\alpha}} \qquad &\hbox{when }  \alpha\in (0,1],\\
\displaystyle k^{-d}\sum_{z\in k^{-1}\Z^d}\left(f(x+z)-f(x)-
\langle \nabla f(x), z\rangle
\right)\frac{w_{kx,k(x+z)}}{|z|^{d+\alpha}}   &\hbox{when } \alpha\in (1,2).
\end{cases}
\end{equation}

\begin{proposition}\label{l3-1} Under {\bf Assumption (H1)}, for every
$f\in C_c^2 (\R^d)$,
there exist a random variable $k_1^*(\w)\ge1$ and a constant $C_5^*>0$ such that for all $k>k_1^*(\w)$,
\begin{equation}\label{l3-1-1}
\int_{k^{-1}\Z^d} |\widehat \sL^{k}f(x)-\bar \sL^{k}f(x) |^2\,\mu^{k}(dx)\le C_5^*\max\{k^{-2(2-\alpha)},k^{-d}\}\log k.
\end{equation}
\end{proposition}
\begin{proof}
 Without loss of generality, we assume that $\text{supp}[ f] \subset B(0,1)$.
 The proof is split into two cases.

\noindent{\bf Case 1: $\alpha\in (1,2).$}\,\,
For every $x,y\in \Z^d$, set
$$\xi_k(x,y):=f(k^{-1}y)-f(k^{-1}x)-k^{-1}
\langle \nabla f(k^{-1}x), (y-x)\rangle,
\quad \eta_k(x,y):=\xi_{k}(x,y)\frac{w_{x,y}-1}{|x-y|^{d+\alpha}}.
$$
Then, for all $x\in \Z^d$,
$$
 \widehat \sL^{k}f(k^{-1}x)=k^{\alpha}
 \sum_{z\in \Z^d \setminus \{x\} }
 \xi_k(x,z)\frac{w_{x,z}}{|x-z|^{d+\alpha}},\quad  \bar \sL^{k}f(k^{-1}x)=k^{\alpha}
 \sum_{z\in \Z^d \setminus \{x\} }
 \xi_k(x,z)\frac{1}{|x-z|^{d+\alpha}}
$$ and
\begin{align}\label{l3-1-5}
\int_{k^{-1}\Z^d} |\widehat \sL^{k}f(x)-\bar \sL^{k}f(x) |^2\,\mu^{k}(dx)
=k^{-(d-2\alpha)}\sum_{x\in \Z^d}\bigg(\sum_{y\in \Z^d\setminus \{x\} }\eta_k(x,y)\bigg)^2.
\end{align}

Set $$b_k(x,y):=\frac{\left|f(k^{-1}y)-f(k^{-1}x)-k^{-1}
\langle \nabla f(k^{-1}x), (y-x)\rangle
\right|}
{|y-x|^{d+\alpha}}=\frac{|\xi_k(x,y)|}{|y-x|^{d+\alpha}}.$$
For every $x\in B_{2k}$,
\begin{equation}\label{l3-1-3a-}\begin{split}
\|b_k(x)\|_2^2:&= \sum_{y\in \Z^d}|b_k(x,y)|^2\le \sum_{y\in B_{3k}}|b_k(x,y)|^2+\sum_{y\in B_{3k}^c}|b_k(x,y)|^2\\
&\le \frac{1}{4}k^{-4}\|\nabla^2 f\|_\infty^2 \sum_{y\in B_{3k}}\frac{1}{|x-y|^{2d+2\alpha-4}}\\
&\quad +
c_1\left(\|f\|_\infty^2+\|\nabla f\|_\infty^2\right)\left(
\sum_{y\in B_{3k}^c}\frac{1}{(1+|y|)^{2d+2\alpha}}+k^{-2}\sum_{z\in B_{3k}^c}\frac{1}{(1+|y|)^{2d+2\alpha-2}}\right)\\
&\le c_2 \max\{k^{-4},k^{-(d+2\alpha)}\}.
\end{split}\end{equation}
When $x\in B_{2k}^c$, we have
\begin{equation}\label{l3-1-4a-}\begin{split}
\|b_k(x)\|_2^2:=\sum_{y\in \Z^d}|b_k(x,y)|^2&=\sum_{y\in B_k}|b_k(x,y)|^2\le c_3\|f\|_\infty^2 \sum_{y\in B_k}\frac{1}{(1+|x|)^{2d+2\alpha}}\le c_4k^d(1+|x|)^{-2d-2\alpha}.
\end{split}\end{equation}

Note that $\Ee[\eta_k(x,z)]=0$ and $\{\eta_k(x,z)\}_{z\neq x}$ is a sequence of independent random variables.
Combining this with
\eqref{l3-1-3a-}
and the Hoeffding inequality (see \cite[Theorem 2.16 on p.\ 21]{BDR},
by taking $Z_{n}=\eta_k(x,y)$ and letting $n \to \infty$ in  (2.31) there)
 yields that for every $\kappa>0$ and $x\in B_{2k}$,
\begin{align*}
 \Pp\left(\left|\sum_{z\in \Z^d}\eta_k(x,z)\right|>\kappa \max\{k^{-2},k^{-d/2-\alpha}\}\log^{1/2} k\right)
 &\le 2\exp\left(- 2\left(\frac{\kappa \max\{k^{-2},k^{-d/2-\alpha}\}\log^{1/2} k}{\|b_k(x)\|_2}\right)^2\right)\\
& \le
 2 \exp\left(-\frac{2\kappa^2}{c_2}\log k\right).
\end{align*}
Similarly, for every $\kappa>0$ and  $x\in B_{2k}^c$,
\begin{align*}
&\Pp\left(\left|\sum_{z\in \Z^d}\eta_k(x,z)\right|>\kappa  k^{d/2}\frac{\log^{1/2}(1+|x|)}{(1+|x|)^{d+\alpha}}\right)\le 2\exp\left(-\frac{2\kappa^2}{c_4}\log (1+|x|)\right).
\end{align*}
In particular, for any $\kappa> \max(\sqrt{c_2d}, \sqrt{c_4d})$ (note that, since
the constants $c_2,c_4$ depend on $f$,  the choice of $\kappa$ depends on $f$ too), it holds that
\begin{align*}
&\sum_{k=1}^\infty \Bigg[\sum_{x\in B_{2k}}\Pp\left(\left|\sum_{z\in \Z^d}\eta_k(x,z)\right|>\kappa \max\{k^{-2},k^{-d/2-\alpha}\}\log^{1/2} k\right)\\
&\qquad+\sum_{x\in B_{2k}^c}\Pp\left(\left|\sum_{z\in \Z^d}\eta_k(x,z)\right|>\kappa k^{d/2}\frac{\log^{1/2}(1+|x|)}{(1+|x|)^{d+\alpha}}\right)\Bigg]<\infty.
\end{align*}
This along with Borel-Cantelli's lemma yields that there is a random variable $k_1(\w)\ge1$ such that for all $k>k_1(\w)$,
$$
\left|\sum_{y\in \Z^d}\eta_k(x,y)\right|\le
\begin{cases}
 \kappa \max\{k^{-2},k^{-d/2-\alpha}\}\log^{1/2} k,&\quad x\in B_{2k},\\
 \kappa k^{d/2}\frac{\log^{1/2}(1+|x|)}{(1+|x|)^{d+\alpha}},&\quad x\in B_{2k}^c.
\end{cases}
$$
Putting this into \eqref{l3-1-5}, we can get the desired assertion \eqref{l3-1-1}.

\ \

\noindent{\bf Case 2: $\alpha\in (0,1].$}\,\,
Set, for every $x,y\in \Z^d$,
\begin{align*}
 \xi_k(x,y):=f(k^{-1}y)-f(k^{-1}x)-k^{-1}
 \langle \nabla f(k^{-1}x), (y-x)\I_{\{|y-x|\le k\}}\rangle
 \quad \eta_k(x,y):=\xi_{k}(x,y)\frac{w_{x,y}-1}{|x-y|^{d+1}}.
\end{align*}
Then, following the same argument as \eqref{l3-1-5}, we have
\begin{align}\label{l3-1-3-2}
\int_{k^{-1}\Z^d}|\widehat \sL^{k}f(x)-\bar \sL^{k}f(x)|^2\,\mu^{k}(dx)
=k^{-(d-2\alpha)}\sum_{x\in \Z^d}\Big(
\sum_{y\in \Z^d\setminus  \{x\}}
\eta_k(x,y)\Big)^2.
\end{align}

Set $$
 b_k(x,y):=\frac{\left|f(k^{-1}y)-f(k^{-1}x)-k^{-1}
 \langle \nabla f(k^{-1}x), (y-x)\I_{\{|y-x|\le k\}}\rangle
 \right|}{|y-x|^{d+\alpha}}.$$
For every $x\in B_{2k}$,
\begin{align*}
  \sum_{y\in \Z^d}|b_k(x,y)|^2
&\le \sum_{|y-x|\le k}|b_k(x,y)|^2+\sum_{|y-x|>k}|b_k(x,y)|^2\\
&\le k^{-4}\|\nabla^2 f\|_\infty^2 \sum_{|y-x|\le k}\frac{1}{|x-y|^{2d+2\alpha-4}}+
c_1 \|f\|_\infty^2
\sum_{|y-x|>k}\frac{1}{(1+|y|)^{2d+2\alpha}}\\
&\le c_2 \max\{k^{-4},k^{-(d+2\alpha)}\}.
\end{align*}
When $x\in B_{2k}^c$,  we have
$$
\sum_{y\in \Z^d}|b_k(x,y)|^2 =\sum_{y\in B_k}|b_k(x,y)|^2
 \le c_3\|f\|_\infty^2 \sum_{y\in B_k}\frac{1}{(1+|x|)^{2d+2\alpha}}\le c_4k^d(1+|x|)^{-2d-2\alpha}.
$$

With aid of the estimate above, we can prove \eqref{l3-1-1} by
the same arguments as these for {\bf Case (1)}.
\end{proof}

Next, we prove the following statement.
\begin{proposition}\label{l3-1-1-1}
Suppose that $\alpha\in (0,1)$.
Under {\bf Assumption (H1)},  for every
 $f\in C_c^2 (\R^d)$,
there exist a random variable $k^*_2(\w)\ge1$ and
a constant $C_6^*>0$ such that for all $k>k^*_2(\w)$,
\begin{equation}\label{l3-1-1-1a}
\int_{k^{-1}\Z^d} |  \sL^{k} f(x)-\bar \sL^{k}f(x) |^2\,\mu^{k}(dx)\le C_6^* \max\{k^{- 2(1-\alpha)},k^{-d}\}\log k.
\end{equation}
\end{proposition}

\begin{proof}
As before we may
assume that $\text{supp} [f]\subset B(0,1)$.
For every $x,y\in \Z^d$, set
$$\xi_k(x,y):=f(k^{-1}y)-f(k^{-1}x),\quad
 \eta_k(x,y):=\xi_{k}(x,y)\frac{w_{x,y}-1}{|x-y|^{d+\alpha}}.
$$
Then, we have
\begin{align}\label{l3-1-2a}
\int_{k^{-1}\Z^d} | \sL^{k} f(x)-\bar \sL^{k}f(x) |^2\,\mu^{k}(dx)
=k^{-(d-2\alpha)}\sum_{x\in \Z^d}\left(
\sum_{z\in \Z^d \setminus \{x\}}
\eta_k(x,z)\right)^2.
\end{align}

We also set
$$
b_k(x,z):=\frac{\left|f(k^{-1}x)-f(k^{-1}z)\right|}{|x-z|^{d+\alpha}}.
$$
For every $x\in B_{2k}$, it holds that
\begin{align*}
\sum_{z\in \Z^d}|b_k(x,z)|^2&\le \sum_{z\in B_{3k}}|b_k(x,z)|^2+\sum_{z\in B_{3k}^c}|b_k(x,z)|^2\\
&\le k^{-2}\|\nabla f\|_\infty^2 \sum_{z\in B_{3k}}\frac{1}{|x-z|^{2d+2\alpha-2}}+
c_1\|f\|_\infty^2 \sum_{z\in B_{3k}^c}\frac{1}{(1+|z|)^{2d+2\alpha}}
\le c_2 \max\{k^{-2},k^{-(d+2\alpha)}\};
\end{align*}
while for $x\in B_{2k}^c$, we have
\begin{align*}
\sum_{z\in \Z^d}|b_k(x,z)|^2&=\sum_{z\in B_k}|b_k(x,z)|^2\le c_3\|f\|_\infty^2 \sum_{z\in B_k}\frac{1}{(1+|x|)^{2d+2\alpha}}\le c_4k^d(1+|x|)^{-2d-2\alpha},
\end{align*}
Hence, we know
\begin{equation}\label{l3-1-3a}
\|b_k(x)\|_2^2:=\sum_{z\in \Z^d}|b_k(x,z)|^2\le
\begin{cases}
c_2 \max\{k^{-2},k^{-(d+2\alpha)}\},&\quad x\in B_{2k},\\
c_4 k^d(1+|x|)^{-2d-2\alpha},&\quad x\in B_{2k}^c.
\end{cases}
\end{equation}

By applying these estimates and following the arguments in {\bf Case (1)} of the proof for Proposition \ref{l3-1}, we can
readily obtain the desired assertion \eqref{l3-1-1-1a}.
  \end{proof}

\begin{remark}The main probabilistic approach to obtain results in Sections \ref{section2}, \ref{section2-} and \ref{section3} (for example, the local Poincar\'e inequality \eqref{l2-1-1},
the $H^{\alpha/2}$-bounds for the solutions of localized Poisson equations, and the
convergence rates of the scaled operators $
\sL^k$) is based on
the
Borel-Cantelli's arguments. Therefore, it is clear that the idea still
works
when the uniformly bounded condition of $w_{x,y}$ in {\bf Assumption (H1)}(3) is
replaced with some moment conditions on $w_{x,y}$.
However,
the rate of convergence may be slower under such conditions.
Indeed, this approach should work
for long range jumps in the one-parameter stationary ergodic environment; see \cite{CCKW2} for more details on homogenization of symmetric non-local Dirichlet forms with $\alpha$-stable-like
jumping kernels. Note that, in the one-parameter stationary ergodic setting, some additional
de-correlation assumptions (for instance,  FKG-type inequality,
 finite range dependence, or negative association)
are needed
in order to get (explicit) moments and concentration estimates for the conductances (see \cite[Section 2]{AH}).
 \end{remark}

\section{Proof of Main Result}

This section
is devoted to the proof of Theorem \ref{T:main}. For every $f\in \mathscr{S}^\lambda_0$,
$\bar u: =\bar R_\lambda f = (\lambda-\bar \sL)^{-1}f \in C_c^\infty(\R^d)$,
and so, by the mean value theorem,
for every $z\in k^{-1}\Z^d$ and $x\in \prod_{1\le i\le d} (z_i,z_i+k^{-1}]$, we can find
a constant $\eta(x,z)\in [0,1]$ such that
\begin{equation}\label{e:pppqqq}\begin{split} \sum_{z\in k^{-1}\Z^d}\int_{\prod_{1\le i\le d} (z_i,z_i+k^{-1}]}
|\bar u(x)-\bar u(z)|^2\,dx
&\le   k^{-2} \sum_{z\in k^{-1}\Z^d}\int_{\prod_{1\le i\le d} (z_i,z_i+k^{-1}]}|\nabla \bar u(x+\eta(x,z)(z-x))|^2\,dx\\
&\le  c_0' k^{-2}
\|\nabla \bar u\|_\infty^2
\le c_0 k^{-2} \end{split}\end{equation} for all $k\ge1$ with some $c_0>0$ (which is independent of $k$).

\begin{proof}[Proof of Theorem $\ref{T:main}$]
Let $\bar u:= \bar R_\lambda f$ and $u_k:=R^{k, \w}_\lambda f$.
Without loss of generality, we assume that $ {\rm supp} [\bar u] \subset B(0,1)$. The proof is split into two cases.

\medskip

 \noindent {\bf Case 1: $\alpha\in (1,2)$.}\,\,  The proof of this case is long
 so we divide it  into six steps.

 \smallskip

{\bf Step 1:}
Let $\phi_m:B_{2^m}\to \R^d$ be the solution to \eqref{e2-1}. We extend it to $\phi_m:\R^d \to \R^d$ by setting $\phi_m(x)=\phi_m(z)$ if $x\in \prod_{1\le i\le d}(z_i,z_i+1]\cap B_{2^m}$ for
some (unique)
$z\in \Z^d\cap B_{2^m}$, and $\phi_m(x)=0$ if $x\notin B_{2^m}$.
For any $k\in \N_+$, let $2^m\le k<2^{m+1}$ with the unique nonnegative integer $m\ge 0$.
According to \eqref{l2-1-1} and \eqref{p2-1-2}, for any $\gamma>0$, there is a random variable $m_1^*:=m_1^*(\w)\ge1$ such that
for all $m>m_1^*$,
\begin{equation}\label{t3-1-2}
\oint_{B_{2^m}}|\phi_m(x)|^2\,\mu(dx)
\le c_12^{m(\alpha-d)}
 \mathscr{E}_{B_{2^m}}^\w(\phi_m,\phi_m)
\le c_2
m^{\frac{\alpha(1+\gamma)}{(2(d-\alpha))\wedge d}}
2^{m\alpha}.
\end{equation}

Define
\begin{equation}\label{eq:v_k}
v_k(x):=\bar u(x)+k^{-1}
\langle \nabla \bar u(x),\phi_{m+2}(k x)\rangle,
\quad x\in \R^d.
\end{equation}
Since $\text{supp} [\bar u ] \subset B_1$, $v_k(x)=0$ for all $x\notin B_1$. Thus, for any $k\ge 2^m>2^{m_1^*}$,
\begin{equation}\label{t3-1-3}
\begin{split}
\int_{\R^d}|v_k(x)-\bar u(x)|^2 \,dx&=k^{-2}
\int_{B_{1}}
|\langle\nabla \bar u(x), \phi_{m+2}(k x)\rangle|^2
\, dx \\
&\le c_3 k^{-2}\int_{B_{1}}|\phi_{m+2}(k x)|^2\,dx=
c_3k^{-(2+d)}\int_{B_{k}}|\phi_{m+2}(x)|^2\,dx\\
&\le c_3k^{-(2+d)}\int_{B_{2^{m+2}}}|\phi_{m+2}|^2(x)\,\mu(dx)
\le c_4k^{-(2-\alpha)}
\log^{\frac{\alpha(1+\gamma)}{(2(d-\alpha))\wedge d}}k,
\end{split}
\end{equation}
where in the last inequality we used \eqref{t3-1-2} and the fact that $2^m\le k<2^{m+1}$.

We also can restrict $v_k:\R^d \to \R$ to $v_k:k^{-1}\Z^d \to \R$. Thus, for any $x\in k^{-1}\Z^d$,
\begin{align*}
\sL^{k}v_k(x)&=\sL^{k}\bar u(x)+k^{-1}
\left\langle\phi_{m+2}(kx),  \int_{k^{-1}\Z^d} \left(\nabla \bar u(x+z)-\nabla \bar u(x)\right)\frac{w_{kx,k(x+z)}}{|z|^{d+\alpha}} \,\mu^{k}(dz) \right\rangle\\
&\quad+k^{-1}\left\langle \nabla \bar u(x),  \int_{k^{-1}\Z^d}\left(\phi_{m+2}\left(k(x+z)\right)-\phi_{m+2}\left(k x\right)\right)
\frac{w_{kx,k(x+z)}}{|z|^{d+\alpha}}\,\mu^{k}(dz) \right\rangle\\
&\quad +k^{-1}\int_{k^{-1}\Z^d}\left\langle \nabla \bar u(x+z)-\nabla \bar u(x)
, \phi_{m+2}(k(x+z))-\phi_{m+2}(kx)\right\rangle\frac{w_{kx,k(x+z)}}{|z|^{d+\alpha}}\,\mu^{k}(dz)\\
&=:\sum_{i=1}^4 I_i^k(x).
\end{align*}

\smallskip

{\bf Step 2:}
By \eqref{e:operator}, \eqref{e:operator0} and \eqref{e2-2}, it holds
\begin{align*}
I_1^k(x)&=\widehat \sL^{k} \bar u(x)+
\left\langle\nabla \bar u(x), \int_{k^{-1}\Z^d}z\frac{w_{kx,k(x+z)}}{|z|^{d+\alpha}}\,\mu^{k}(dz) \right\rangle\\
&= \widehat \sL^{k} \bar u(x)+k^{\alpha-1}
\left\langle\nabla \bar u(x),  \int_{\Z^d}z\frac{w_{k x,k x+z}}{|z|^{d+\alpha}}\,\mu(dz) \right\rangle
= \widehat \sL^{k} \bar u(x)+k^{\alpha-1}\langle \nabla \bar u(x), V(k x)\rangle.
\end{align*}
According to
Propositions \ref{l3-1-0} and \ref{l3-1},
we can find a random  variable $m_2^*:=m_2^*(\w)\ge1$ such that for every $k\ge 2^m>2^{m^*_2}$,
$$
\int_{k^{-1}\Z^d}|\widehat \sL^{k}\bar u(x)-\bar \sL \bar u(x)|^2\, \mu^{k}(dx)\le c_5\max\{k^{-2(2-\alpha)},k^{-d}\}\log k.
$$
Thus we have
\begin{equation}\label{t3-1-3a}
I_1^k(x)=\bar \sL \bar u(x)+k^{\alpha-1}
\left\langle\nabla \bar u(x), V(k x)\right\rangle
+K_1^k(x),
\end{equation}
where for all $k>2^{m_2^*}$,
$$
\int_{k^{-1}\Z^d}|K_1^k(x)|^2 \,\mu^{k}(dx)\le c_5
\max\{k^{-2(2-\alpha)},k^{-d}\}\log k.
$$

\smallskip

{\bf Step 3:}
We rewrite
\begin{align*}
I_2^k(x)&=k^{-1}
\Bigg\langle\phi_{m+2}(k x),
\int_{k^{-1}\Z^d} \left(\nabla \bar u(x+z)-\nabla \bar u(x)-\langle \nabla^2 \bar u(x), z\I_{\{|z|\le 1\}}\rangle\right)\frac{w_{kx,k(x+z)}}{|z|^{d+\alpha}}\,\mu^{k}(dz)\\
&\qquad\qquad\qquad\qquad\,\,+\int_{k^{-1}\Z^d}\left\langle\nabla^2 \bar u(x), z\I_{\{|z|\le 1\}}\right\rangle\frac{w_{kx,k(x+z)}}{|z|^{d+\alpha}}\,
\mu^{k}(dz)\Bigg\rangle\\
&=:I_{21}^k(x)+I_{22}^k(x).
\end{align*}

Since $\bar u\in C_c^\infty(\R^d)$, we get $|I_{21}^k(x)|\le c_6k^{-1}|\phi_{m+2}(k x)|$. So, by \eqref{t3-1-2} and the fact that $2^m\le k<2^{m+1}$, it holds
that for any $k>2^{m_1^*}$,
\begin{equation*}
\int_{k^{-1}\Z^d}|I_{21}^k(x)|^2 \,\mu^{k}(dx)\le c_7k^{-2}\oint_{B_{2^{m+2}}}|\phi_{m+2}(x)|^2 \,\mu(dx)\le
c_8k^{-(2-\alpha)}
\log^{\frac{\alpha(1+\gamma)}{(2(d-\alpha))\wedge d}}k.
\end{equation*}

For every $g\in L^2(k^{-1}\Z^d;\mu^{k})$,  we have
\begin{align*}
&\int_{k^{-1}\Z^d}I_{22}^k(x)g(x)\,\mu^{k}(dx)\\
&=k^{-1}\int_{k^{-1}\Z^d}\int_{\{|z|\le 1\}}
\langle G(x),\phi_{m+2}(kx)\otimes z\rangle
\frac{w_{k x, k(x+z)}}{|z|^{d+\alpha}}\,\mu^{k}(dz)\,\mu^{k}(dx)\\
&=k^{-1}\Bigg(\frac{1}{2}\int_{k^{-1}\Z^d}\int_{\{|z|\le 1\}}
\langle G(x),\phi_{m+2}(kx)\otimes z\rangle
\frac{w_{k x, k(x+z)}}{|z|^{d+\alpha}}\mu^{k}(dz)\,\mu^{k}(dx)\\
&\qquad\qquad -\frac{1}{2}\int_{k^{-1}\Z^d}\int_{\{|z|\le 1\}}
\langle G(x+z),\phi_{m+2}(k(x+z))\otimes z\rangle
\frac{w_{k x, k(x+z)}}{|z|^{d+\alpha}}\,\mu^{k}(dz)\,\mu^{k}(dx)\Bigg)\\
&=(2k)^{-1}\int_{k^{-1}\Z^d}\int_{\{|z|\le 1\}}
\left\langle G(x),\phi_{m+2}(kx)\otimes z\right\rangle- \left\langle G(x+z),\phi_{m+2}(k(x+z))\otimes z\right\rangle
\frac{w_{k x, k(x+z)}}{|z|^{d+\alpha}}\,\mu^{k}(dz)\,\mu^{k}(dx),
\end{align*}
where $G(x):=g(x)\nabla^2 \bar u(x)$ and in the second equality we have used the change of variables $x=\tilde x+\tilde z$ and $z=-\tilde z$.
Next, we set
\begin{align*}
&\left|\int_{k^{-1}\Z^d}I_{22}^k(x)g(x)\,\mu^{k}(dx)\right|\\
&\le (2k)^{-1}\Bigg(\int_{B_2}\int_{\{|z|\le 1\}}|\phi_{m+2}(k x)|
|G(x+z)-G(x)||z|\frac{w_{k x,k(x+z)}}{|z|^{d+\alpha}}\,\mu^{k}(dz)\,\mu^{k}(dx)\\
&\qquad \qquad +\int_{B_1}\int_{\{|z|\le 1\}}|G(x)|
|\phi_{m+2}(k (x+z))-\phi_{m+2}(k x)||z|\frac{w_{k x,k(x+z)}}{|z|^{d+\alpha}}\,\mu^{k}(dz)\,\mu^{k}(dx)\Bigg)\\
&=:I_{221}^k+I_{222}^k,
\end{align*}
where we applied the change of variables $x=\tilde x+\tilde z$, $z=-\tilde z$ as before
and used the fact that $\text{supp} [G] \subset B_1$ and $\text{supp} [ \phi_{m+2}] \subset B_{2k}$.

For any
$G=\{G^{(ij)}\}_{1\le i,j\le d}:k^{-1}\Z^d \to \R^{d}\times \R^d$
with $G^{(ij)}\in L^2(k^{-1}\Z^d;\mu^k)$ for every $1\le i,j \le d$,
define
\begin{align*}
 \mathscr{E}^{k,\w}(G,G) &:
 =\frac{1}{2}
 \sum_{i,j=1}^d
 \int_{k^{-1}\Z^d}\int_{k^{-1}\Z^d}
\frac{(G^{(ij)}(x+z)-G^{(ij)}(x))^2w_{k x,k(x+z)}(\w)}{|z|^{d+\alpha}}\,\mu^{k}(dz)\,\mu^{k}(dx).
\end{align*}
According to the Cauchy-Schwartz inequality, we have for all $k>2^{m_1^*}$,
\begin{align*}
I_{221}^k &\le c_9k^{-1}
\left(\int_{B_2}|\phi_{m+2}(k x)|^2\left(\int_{\{|z|\le 1\}}\frac{|z|^2}{|z|^{d+\alpha}}\,\mu^{k}(dz)\right)
\,\mu^{k}(dx)\right)^{1/2}\\
&\qquad\qquad\times\left(\int_{k^{-1}\Z^d}\int_{\{|z|\le 1\}}\frac{|G(x+z)-G(x)|^2w_{k x,k(x+z)}}{|z|^{d+\alpha}}
\,\mu^{k}(dz)\,\mu^{k}(dx)\right)^{1/2}\\
&\le c_{10}k^{-1}\left(\int_{B_2}|\phi_{m+2}(k x)|^2\,\mu^{k}(dx)\right)^{1/2}  \mathscr{E}^{k,\w}(G,G) ^{1/2}\\
&\le c_{11}k^{-(2-\alpha)/2}
(\log^{\frac{\alpha(1+\gamma)}{(4(d-\alpha))\wedge (2d)}}k)
\,\, \mathscr{E}^{k,\w}(G,G) ^{1/2}\\
&\le c_{12}
k^{-(2-\alpha)/2}
(\log^{\frac{\alpha(1+\gamma)}{(4(d-\alpha))\wedge (2d)}}k)
\left( \mathscr{E}^{k,\w}(g,g)^{1/2}+\|g\|_{L^2(k^{-1}\Z^d;\mu^{k})}\right),
\end{align*}
where in the third inequality we used
\eqref{t3-1-2} and $2^m\le k <2^{m+1}$, and the last inequality follows from the definition of  $\mathscr{E}^{k,\w}(G,G)$.

On the other hand, by the Cauchy-Schwartz inequality again,
\begin{equation}\label{t3-1-4}
\begin{split}
I_{222}^k
&\le c_{13}k^{-1}\left(\int_{B_1}G^2(x)\left(\int_{\{|z|\le 1\}}\frac{|z|^2}{|z|^{d+\alpha}}\,\mu^{k}(dz)\right)\,
\mu^{k}(dx)\right)^{1/2}\\
&\quad\times\left(\int_{B_1}\int_{\{|z|\le 1\}}
\left|\phi_{m+2}(k(x+z))-\phi_{m+2}(k x)\right|^2
\frac{w_{k x,k(x+z)}}{|z|^{d+\alpha}}\,\mu^{k}(dz)\,\mu^{k}(dx)\right)^{1/2}.
\end{split}
\end{equation}
By \eqref{p2-1-2} we have for all $k>2^{m_2^*}$,
\begin{align*}
&\int_{B_1}\int_{\{|z|\le 1\}}
\left|\phi_{m+2}(k(x+z))-\phi_{m+2}(k x)\right|^2
\frac{w_{k x,k(x+z)}}{|z|^{d+\alpha}}\,\mu^{k}(dz)\,\mu^{k}(dx)\\
&\le k^{-(d-\alpha)}
\int_{B_{k}}\int_{B_{2k}}\left|\phi_{m+2}(x)-\phi_{m+2}(y)\right|^2
\frac{w_{x,y}}{|x-y|^{d+\alpha}}\,\mu(dy)\,\mu(dx)\\
&\le k^{-(d-\alpha)}
\int_{B_{2k}}\int_{B_{2k}}\left|\phi_{m+2}(y)-\phi_{m+2}(x)\right|^2
\frac{w_{x,y}}{|x-y|^{d+\alpha}}\,
\mu(dy)\,\mu(dx)\\
&\le k^{-(d-\alpha)}
\mathscr{E}_{B_{2k}}^\w(\phi_{m+2},\phi_{m+2})
\le c_{14}k^{\alpha}
\log^{\frac{\alpha(1+\gamma)}{(2(d-\alpha))\wedge d}}k,
\end{align*}
Putting this into \eqref{t3-1-4} yields that for all $k>2^{m_2^*}$,
\begin{align*}
I_{222}^k \le c_{15}k^{-(2-\alpha)/2}
(\log^{\frac{\alpha(1+\gamma)}{(4(d-\alpha))\wedge (2d)}}k)
\|G\|_{L^2((B_1;\R^d\times \R^d);\mu^{k})}\le c_{16}
k^{-(2-\alpha)/2}
(\log^{\frac{\alpha(1+\gamma)}{(4(d-\alpha))\wedge (2d)}}k)
\|g\|_{L^2(B_1;\mu^{k})}.
\end{align*}

Combining
all the above estimates, we have
\begin{equation}\label{t3-1-5}
I_2^k(x)=K_{21}^k(x)+K_{22}^k(x),
\end{equation}
where for every $k>2^{m_2^*}$,
$$\int_{k^{-1}\Z^d}|K_{21}^k(x)|^2\,\mu^{k}(dx)\le c_{17}k^{-(2-\alpha)}
\log^{\frac{\alpha(1+\gamma)}{(2(d-\alpha))\wedge d}}k
$$ and
\begin{align*}
\left|\int_{k^{-1}\Z^d}K_{22}^k(x)g(x)\,\mu^{k}(dx)\right|\le c_{18}
k^{-(2-\alpha)/2}
(\log^{\frac{\alpha(1+\gamma)}{(4(d-\alpha))\wedge (2d)}}k)
(\E^{k,\w}(g,g)^{1/2}+\|g\|_{L^2(k^{-1}\Z^d;\mu^{k})} )
\end{align*}
for all $g\in L^2(k^{-1}\Z^d;\mu^{k})$.

\smallskip

{\bf Step 4:}
We know that
\begin{align*}
I_3^k(x)&=k^{\alpha-1}
\left\langle\nabla \bar u(x),  \int_{\Z^d}\left(\phi_{m+2}(y)-\phi_{m+2}(k x)\right)\frac{w_{kx,y}}{|y-k x|^{d+\alpha}}\,\mu(dy) \right\rangle\\
&=k^{\alpha-1}\Bigg\langle \nabla \bar u(x),  \int_{B_{2^{m+2}}}\left(\phi_{m+2}(y)-\phi_{m+2}(k x)\right)\frac{w_{kx,y}}
{|y-k x|^{d+\alpha}}\,\mu(dy)\\
&\qquad\qquad\qquad \qquad\,\,+\int_{B_{2^{m+2}}^c}\left(\phi_{m+2}(y)-\phi_{m+2}(k x)\right)\frac{w_{kx,y}}
{|y-k x|^{d+\alpha}}\,\mu(dy)\Bigg\rangle\\
&=k^{\alpha-1}\left\langle \nabla \bar u(x),\sL_{B_{2^{m+2}}}\phi_{m+2}(k x)+\int_{B_{2^{m+2}}^c}\left(\phi_{m+2}(y)-\phi_{m+2}(k x)\right)
\frac{w_{kx,y}}{|y-k x|^{d+\alpha}}\,\mu(dy) \right\rangle\\
&=k^{\alpha-1}\left\langle\nabla \bar u(x),-V(k x)+\oint_{B_{2^{m+2}}}V\,d\mu+\int_{B_{2^{m+2}}^c}\left(\phi_{m+2}(y)-\phi_{m+2}(k x)\right)\frac{w_{kx,y}}{|y-k x|^{d+\alpha}}\,\mu(dy) \right\rangle\\
&=:-k^{\alpha-1}\langle\nabla \bar u(x), V(k x)\rangle+I_{31}^k(x)+I_{32}^k(x).
\end{align*}
Combining \eqref{p2-1-4} with the fact that $2^m\le k<2^{m+1}$, for every
$\theta\in (\alpha/d,1)$,
we can find a random variable $m_3^*:=m_3^*(\w)>1$ such that
for all $k>2^{m_3^*}$,
$$
\int_{k^{-1}\Z^d}|I_{31}^k(x)|^2\,\mu^{k}(dx) \le k^{2(\alpha-1)}\|\nabla \bar u\|_\infty^2\int_{B_1}\Big|\oint_{B_{2^{m+2}}}V\,d\mu\Big|^2 \,
\mu^k(dx)
\le c_{19}k^{2\alpha-2-\theta d}.
$$
Since $\phi_{m+2}(y)=0$ for all $y\in B_{2^{m+2}}^c$ and
$\text{supp} [ u ] \subset B_1$, it holds that for all $k>2^{m_3^*}$,
\begin{align*}
\int_{k^{-1}\Z^d}|I_{32}^k(x)|^2\,\mu^{k}(dx)&\le
k^{2(\alpha-1)}\|\nabla \bar u\|_\infty^2
\int_{B_1}|\phi_{m+2}(k x)|^2 \bigg(\sum_{y\in k^{-1}\Z^d:|y-k x|\ge k}\frac{w_{kx,y}}{|y-k x|^{d+\alpha}}\bigg)^2\,
\mu^{k}(dx)\\
&\le c_{20}k^{-2}\int_{B_1}|\phi_{m+2}(k x)|^2\,\mu^{k}(dx)\le c_{21}k^{-(2-\alpha)}
\log^{\frac{\alpha(1+\gamma)}{(2(d-\alpha))\wedge d}}k,
\end{align*}
where the last inequality follows from  \eqref{t3-1-2}.

Summarizing all above estimates for $I_3^k$, we
conclude that
\begin{equation}\label{t3-1-6}
\begin{split}
I_3^k(x)&=-k^{\alpha-1}
\langle \nabla \bar u(x), V(kx)\rangle
+K_3^k(x),
\end{split}
\end{equation}
where for all $k>2^{\max\{m_1^*, m_3^*\}}$, $$\int_{k^{-1}\Z^d}|K_3^k(x)|^2\,\mu^{k}(dx)\le c_{22}k^{-(2-\alpha)}
\log^{\frac{\alpha(1+\gamma)}{(2(d-\alpha))\wedge d}}k.
$$
Here,  we used the fact that
$-(2-\alpha)>2\alpha-2-\theta d$
as
$\theta\in (\alpha/d,1)$.

\smallskip

{\bf Step  5:}
When $x\in B_{3/2}^c$, we get from the fact $\nabla \bar u(z)\neq 0$ only if $z\in B_1$ that
for every $k>2^{m_1^*}$,
\begin{align*}
|I_4^k(x)|&=k^{-1}\left|\int_{\{x+z\in B_1\}}\frac{
\left\langle \nabla \bar u(x+z), \phi_{m+2}\left(k(x+z)\right)-\phi_{m+2}(k x)\right\rangle
w_{k x,k(x+z)}}{|z|^{d+\alpha}}\,\mu^{k}(dz)\right|\\
&=k^{-1}\left|\int_{B_1}\frac{
\left\langle \nabla \bar u(y),
\phi_{m+2}\left(k y\right)-\phi_{m+2}\left(k x\right)\right\rangle
w_{k x,k y}}{|y-x|^{d+\alpha}}\,\mu^{k}(dy)\right|\\
&\le c_{23}k^{-1}(1+|x|)^{-d-\alpha}
\left(\int_{ B_1}\left|\phi_{m+2}\left(k y\right)\right|\mu^{k}(dy)+|\phi_{m+2}(k x)|\right)\\
&\le c_{24}k^{-1}(1+|x|)^{-d-\alpha}
\left(\left(\int_{  B_1}\left|\phi_{m+2}\left(k y\right)\right|^2\mu^{k}(dy)\right)^{1/2}+|\phi_{m+2}(k x)|\right)\\
&\le c_{25}\left(k^{-(2-\alpha)/2}
(\log^{\frac{\alpha(1+\gamma)}{(4(d-\alpha))\wedge (2d)}}k)
\,(1+|x|)^{-d-\alpha}+k^{-1}|\phi_{m+2}(k x)|\right).
\end{align*}
Here the first inequality
follows from the fact that $|y-x|\ge c_{26}(1+|x|)$ for all $y\in B_1$ and $x\in B_{3/2}^c$,
and in the last inequality we used \eqref{t3-1-2}.

When $x\in B_{3/2}$, it holds that
\begin{align*}
|I_4^k(x)|&=k^{-1}\Bigg|
\int_{k^{-1}B_{{2^{m+2}}}}
\left\langle \nabla \bar u(y)-\nabla \bar u(x), \phi_{m+2}\left(k y\right)-\phi_{m+2}\left(k x\right)\right\rangle
\frac{w_{k x,k y}}{|y-x|^{d+\alpha}}\,\mu^{k}(dy)\\
&\qquad \quad+
\langle \nabla \bar u(x), \phi_{m+2}(k x)\rangle
\int_{B_{k^{-1}{2^{m+2}}}^c}
\frac{w_{k x,k y}}{|y-x|^{d+\alpha}}\,\mu^{k}(dy)\Bigg|\\
&\le c_{26}k^{-1}\left(\int_{B_{k^{-1}2^{m+2}}}
\frac{\left|\nabla \bar u\left(y\right)-\nabla \bar u\left(x\right)\right|^2}{|y-x|^{d+\alpha}}\,\mu^{k}(dy)\right)^{1/2}\\
&\qquad\qquad\times
\left(\int_{B_{k^{-1}2^{m+2}}}
\frac{\left|\phi_{m+2}\left(k y\right)-\phi_{m+2}\left(k x\right)\right|^2w_{k x,k y}}{|y-x|^{d+\alpha}}\,\mu^{k}(dy)
\right)^{1/2}\\
&\quad +c_{26}k^{-1}|\phi_{m+2}(k x)|\cdot\left(\int_{\{y\in k^{-1}\Z^d:|y-x|\ge 1/2\}}\frac{1}{|y-x|^{d+\alpha}}\,\mu^{k}(dy)\right)\\
&\le c_{27}k^{-1}\left(\left(\int_{B_{k^{-1}2^{m+2}}}
\frac{\left|\phi_{m+2}\left(k y\right)-\phi_{m+2}\left(k x\right)\right|^2w_{k x,k y}}{|y-x|^{d+\alpha}}\,\mu^{k}(dy)
\right)^{1/2}+|\phi_{m+2}(k x)|\right).
\end{align*}
Here the equality above
is due to the fact $\nabla \bar u(y)=\phi_{m+2}(k y)=0$ for all $y\in B_{k^{-1}{2^{m+2}}}^c$,
while
in the first
inequality we used the Cauchy-Schwarz inequality and the fact $|y-x|\ge 1/2$ for every $y\in B_{k^{-1}{2^{m+2}}}^c$ and $x\in B_{3/2}$.

Combining all above estimate yields that for all $k>2^{m_1^*}$,
\begin{align*}
|I_4^k(x)|&\le c_{28}\Bigg(k^{-1}\left(\int_{B_{k^{-1}2^{m+2}}}
\frac{\left|\phi_{m+2}\left(k y\right)-\phi_{m+2}\left(k x\right)\right|^2w_{k x, k y}}{|y-x|^{d+\alpha}}\,\mu^{k}(dy)
\right)^{1/2}\I_{B_{3/2}}(x)\\
&\qquad\quad +k^{-(2-\alpha)/2}
(\log^{\frac{\alpha(1+\gamma)}{(4(d-\alpha))\wedge (2d)}}k)
(1+|x|)^{-d-\alpha}\I_{B_{3/2}^c}(x)
+k^{-1}|\phi_{m+2}(k x)|\Bigg).
\end{align*}
Then, applying the Cauchy-Schwarz inequality again, we obtain that for all $k>2^{\max\{m_1^*,m_2^*\}}$,
\begin{equation}\label{t3-1-7}
\begin{split}
&\int_{k^{-1}\Z^d}|I_4^k(x)|^2\,\mu^{k}(dx)\\
&\le
c_{29}k^{-2}\int_{B_{3/2}}\int_{k^{-1}B_{2^{m+2}}}
\frac{\left|\phi_{m+2}\left(k y\right)-\phi_{m+2}\left(k x\right)\right|^2w_{k x,k y}}{|y-x|^{d+\alpha}}\,\mu^{k}(dy)
\,\mu^{k}(dx)\\
&\quad +c_{29}k^{-(2-\alpha)}
(\log^{\frac{\alpha(1+\gamma)}{(2(d-\alpha))\wedge d}}k)
\int_{B_{3/2}^c}\frac{1}{(1+|x|)^{2d+2\alpha}}\,\mu^{k}(dx)
+c_{29}k^{-2}\int_{B_{k^{-1}2^{m+2}}}|\phi_{m+2}(k x)|^2\,\mu^{k}(dx)\\
&\le c_{30}k^{-(2-\alpha)-d}\mathscr{E}_{B_{2^{m+2}}}^\w(\phi_{m+2},\phi_{m+2})+c_{29}k^{-(2-\alpha)}
(\log^{\frac{\alpha(1+\gamma)}{(2(d-\alpha))\wedge d}}k)
\int_{B_{3/2}^c}\frac{1}{(1+|x|)^{2d+2\alpha}}\,\mu^{k}(dx)\\
&\quad +c_{30}k^{-2}\oint_{B_{2^{m+2}}}|\phi_{m+2}(x)|^2\,\mu(dx)\\
&\le c_{31}k^{-(2-\alpha)}
\log^{\frac{\alpha(1+\gamma)}{(2(d-\alpha))\wedge d}}k,
\end{split}
\end{equation}
where the last inequality is due to \eqref{p2-1-2} and \eqref{t3-1-2}.

\smallskip

{\bf Step 6:}
Putting \eqref{t3-1-3a}, \eqref{t3-1-5}, \eqref{t3-1-6} and \eqref{t3-1-7} together, we find that
\begin{equation}\label{t3-1-8}
\begin{split}
\sL^{k}v_k(x)=\bar \sL \bar u(x)
+J_1^k(x)+J_2^k(x),\quad x\in k^{-1}\Z^d,
\end{split}
\end{equation}
where
\begin{equation}\label{t3-1-9}
\begin{split}
&\left|\int_{k^{-1}\Z^d}J_1^k(x)g(x)\,\mu^{k}(dx)\right|\le c_{29}k^{-(2-\alpha)/2}
(\log^{\frac{\alpha(1+\gamma)}{(4(d-\alpha))\wedge (2d)}}k)
(\E^{k,\w}(g,g)^{1/2}+\|g\|_{L^2(k^{-1}\Z^d;\,\mu^{k})} ),\\
&\int_{k^{-1}\Z^d}|J_2^k(x)|^2\,\mu^{k}(dx)\le c_{28}\left(\max\{k^{-2(2-\alpha)},k^{-d}\}\log k + k^{-(2-\alpha)}
\log^{\frac{\alpha(1+\gamma)}{(2(d-\alpha))\wedge d}}k
\right)
\end{split}
\end{equation} hold for all $k> k^*_0:=2^{\max\{m_2^*,m_3^*\}}$ and $g\in L^2(k^{-1}\Z^d;\mu^{k})$.
Therefore,
\begin{align*}
&\lambda(u_k(x)-v_k(x))-\sL^{k}(u_k-v_k)(x)\\
&= (\lambda u_k(x)-\sL^{k}u_k(x) )+
\lambda(\bar u(x)-v_k(x))- (\lambda \bar u(x)-
\bar \sL \bar u(x) )+J_1^k(x)+J_2^k(x)\\
&=J_1^k(x)+J_2^k(x)+\lambda (\bar u(x)-v_k(x))\\
&=:J_1^k(x)+Q^k(x).
\end{align*}
 We note that, due to \eqref{t3-1-3} and \eqref{t3-1-9}, for all
$k>k^*_0(\lambda):=k^*_0\vee c_\lambda$,
\begin{equation}\label{t3-1-10}
\int_{k^{-1}\Z^d}|Q^k(x)|^2\,\mu^{k}(dx)\le c_{30}
k^{-(2-\alpha)}
\log^{\frac{\alpha(1+\gamma)}{(2(d-\alpha))\wedge d}}k,
\end{equation} where in the inequality above we used the fact that $2-\alpha<d$ for all $\alpha\in (1,2)$.
Note that here and below the constants may depend on $\lambda$.

Since $u_k,v_k\in L^2(k^{-1}\Z^d;\mu^k)$,
multiplying $u_k-v_k$
on
both sides of the equality above, integrating  with respect to $\mu^{k}$
and applying \eqref{e1-1a},
 we have   for all
$k>k^*_0(\lambda)$,
\begin{align*}
&\lambda\|u_k-v_k\|^2_{L^2(k^{-1}\Z^d;\mu^{k})}+ \E^{k,\w}(u_k-v_k, u_k-v_k)\\
&=\lambda \|u_k-v_k\|^2_{L^2(k^{-1}\Z^d;\mu^{k})}-
\int_{k^{-1}\Z^d}\sL^{k}(u_k-v_k)(x)\cdot\left(u_k-v_k\right)(x)\,\mu^{k}(dx)\\
&=\int_{k^{-1}\Z^d}J_1^k(x)\left(u_k-v_k\right)(x)\,\mu^{k}(dx)+
\int_{k^{-1}\Z^d}Q^k(x)\cdot\left(u_k-v_k\right)(x)\,\mu^{k}(dx)\\
&\le \int_{k^{-1}\Z^d}J_1^k(x)\cdot\left(u_k-v_k\right)(x)\,\mu^{k}(dx)+2\lambda^{-1}
\int_{k^{-1}\Z^d}|Q^k(x)|^2\,\mu^{k}(dx)+\frac{\lambda}{2}\|u_k-v_k\|^2_{L^2(k^{-1}\Z^d;\mu^{k})}\\
&\le c_{31}k^{-(1-\alpha/2)}
(\log^{\frac{\alpha(1+\gamma)}{(4(d-\alpha))\wedge (2d)}}k)
\big(\E^{k,\w}(u_k-v_k, u_k-v_k)^{1/2}+\|u_k-v_k\|_{L^2(k^{-1}\Z^d;\mu^{k})}\big)\\
&+c_{31}k^{-(2-\alpha)}
(\log^{\frac{\alpha(1+\gamma)}{(2(d-\alpha))\wedge d}}k)
+\frac{\lambda}{2}\|u_k-v_k\|^2_{L^2(k^{-1}\Z^d;\mu^{k})}\\
&\le \frac{3\lambda}{4}\|u_k-v_k\|^2_{L^2(k^{-1}\Z^d;\mu^{k})}+\frac{3}{4}
\E^{k,\w}(u_k-v_k, u_k-v_k)+c_{32}k^{-(2-\alpha)}
\log^{\frac{\alpha(1+\gamma)}{(2(d-\alpha))\wedge d}}k,
\end{align*}
where in the first and the last inequalities we have used Young's inequality, and  the second inequality follows from
\eqref{t3-1-9} and \eqref{t3-1-10}.
From the estimate above, we immediately get that
\begin{align*}
\|u_k-v_k\|^2_{L^2(k^{-1}\Z^d;\mu^{k})}\le c_{33}k^{-(2-\alpha)}
\log^{\frac{\alpha(1+\gamma)}{(2(d-\alpha))\wedge d}}k.
\end{align*}
Therefore it holds that for all $k>k^*_0(\lambda) $,
\begin{equation}\label{t3-1-12}
\begin{split}
 \|u_k-v_k\|^2_{L^2(\R^d;dx)}
&=\sum_{z\in k^{-1}\Z^d}\int_{\prod_{1\le i\le d}(z_i,z_i+k^{-1}]}|u_k(x)-v_k(x)|^2\,dx\\
&\le 2\sum_{z\in k^{-1}\Z^d}\int_{\prod_{1\le i\le d}(z_i,z_i+k^{-1}]}\left(|u_k(x)-v_k(z)|^2+
|v_k(x)-v_k(z)|^2\right)\,dx\\
&=2\|u_k-v_k\|^2_{L^2(k^{-1}\Z^d;\mu^{k})}+2\sum_{z\in k^{-1}\Z^d}\int_{\prod_{1\le i\le d}(z_i,z_i+k^{-1}]}
|v_k(x)-v_k(z)|^2\,dx\\
&\le 2\|u_k-v_k\|^2_{L^2(k^{-1}\Z^d;\mu^{k})}+c_{34}\bigg(k^{-2}+k^{-2}\sum_{z\in k^{-1}\Z^d}k^{-d}|\phi_{m+2}(k z)|^2\bigg)\\
&\le 2\|u_k-v_k\|^2_{L^2(k^{-1}\Z^d;\mu^{k})}+c_{35}\bigg(k^{-2}+k^{-2}\oint_{B_{2^{m+2}}}|\phi_{m+2}(z)|^2\bigg)\\
&\le c_{36}k^{-(2-\alpha)}
\log^{\frac{\alpha(1+\gamma)}{(2(d-\alpha))\wedge d}}k,
\end{split}
\end{equation}
where the second inequality
is from the following estimate due to \eqref{eq:v_k}
\begin{align*}
|v_k(x)-v_k(z)|&\le c_{37}\big(|\bar u(x)-\bar u(z)|+
k^{-1}|\nabla \bar u(x)-\nabla \bar u(z)|\cdot |\phi_{m+2}(kz)|\big)\\
&\le c_{38}k^{-1}\left(1+|\phi_{m+2}(kz)|\right),
\quad z\in 2^{-1}\Z^d, x\in \prod_{1\le i\le d}(z_i,z_i+k^{-1}]
\end{align*}
and \eqref{e:pppqqq}.
Note that in the first inequality above, we used the fact
$\phi_{m+2}(kx)=\phi_{m+2}(kz)$ for every  $x\in \prod_{1\le i\le d}(z_i,z_i+k^{-1}]$
by the way of extending the function at the beginning of {\bf Step 1}.

Combining this with \eqref{t3-1-3} we complete the proof for the case
of
 $\alpha\in (1,2)$.

\medskip

{\bf Case 2: $\alpha\in (0,1]$.}\,\, The proof is almost the same as that of the case $\alpha\in (1,2)$. Without loss of generality we assume that $\text{supp} [u ]  \subset B(0,1)$.
Let $\widetilde \phi_m:B_{2^m}\cap \Z^d \to \R^d$ be the solution to \eqref{e2-1-2--}.
For every $k\ge 1$, we choose $m\in \mathbb{N}$ such that $2^m\le k<2^{m+1}$, and define
$$v_k(x):=\bar u(x)+k^{-1}
 \langle \nabla \bar u(x), \widetilde \phi_{m+2}(k x) \rangle
,\quad x\in \R^d,$$ where we
extend $\widetilde \phi_m:\R^d \to \R^d$ by setting $\widetilde \phi_m(x)=\widetilde \phi_m(z)$ if $x\in \prod_{1\le i\le d} (z_i,z_i+1]\cap B_{2^m}$ for
some $z\in \Z^d\cap B_{2^m}$ and setting
$\widetilde \phi_m(x)=0$ if $x\notin B_{2^m}$.
Hence, according to \eqref{l2-1-1} and \eqref{p2-1-2-}, there is a random variable $m_4^*:=m_4^*(\w)\ge1$  such that
for all $m>m_4^*$,
\begin{equation}\label{a31}
\oint_{B_{2^m}}
|\widetilde \phi_m(x)|^2\,\mu(dx)
\le c_{38} 2^{m(\alpha -d)}
 \mathscr{E}_{B_{2^m}}(\widetilde \phi_m,\widetilde \phi_m)
\le c_{39} \begin{cases}2^{m(2-\alpha)}
m^{\frac{\alpha(1+\gamma)}{2(d-\alpha)\wedge d}}
&\quad \alpha\in (0,1),\\
 2^{m}
m^{2+\frac{1+\gamma}{2(d-1)\wedge d}},
&\quad \alpha=1.\end{cases}
\end{equation}

On the other hand, according to  \eqref{l3-1-1} and \eqref{l3-1-2}, there are a constant $c_{40}>0$ and $k_0^*(\w)\ge1 $ such that for all $k> k_0^*(\w)$,
\begin{equation*}
\int_{k^{-1}\Z^d}|\widehat \sL^k \bar u(x)-\bar \sL \bar u(x)|^2\mu^k(dx)\le c_{40}\begin{cases} \max\{k^{-2},k^{-d}\log k\},&\quad \alpha\in (0,1),\\
\max\{k^{-2} \log ^2 k , k^{-d}\log k\},&\quad \alpha=1.\end{cases}\end{equation*} This along with \eqref{a31}
and the argument for the case that $\alpha\in (1,2)$ yields the assertion for $\alpha=1$ in \eqref{t3-1-1-1} and that, for $\alpha\in (0,1)$, there are
constants
$c_{41}>0$ and $k_0(\w)\ge1$ so that
\begin{equation}\label{t3-1-11}
\|u_k-\bar u\|_{L^2(\R^d;dx)}\le c_{41}k^{-\alpha/2}
\log^{\frac{\alpha(1+\gamma)}{4(d-\alpha)\wedge (2d)}}k,
\quad  k>k_0(\w).
\end{equation}

\medskip

Now, we turn to the remaining estimate for $\alpha\in (0,1)$ in \eqref{t3-1-1-1}.
For any $x\in k^{-1}\Z^d$,
\begin{equation}\label{e:proof1}
\begin{split}
\lambda(u_k(x)-\bar u(x))-\sL^{k}(u_k-\bar u)(x)
&= (\lambda u_k (x)- \sL^{k}u_k(x))-(\lambda \bar u(x)- \bar \sL \bar u(x))-\bar \sL \bar u(x)+ \sL^{k} \bar u(x)\\
&=-\bar \sL \bar u(x)+ \sL^{k} \bar u(x)=:J_1^k(x).
\end{split}
\end{equation}
According to
Proposition \ref{l3-1-1-1} and \eqref{l3-1-2},
there is a random variable $k_1(\w)\ge1$ such that for all $k> k_1(\w)$,
$$\int_{k^{-1}\Z^d} |J_1^k(x)|^2\,\mu^{k}(dx)\le c_{42} \max\{k^{-2(1-\alpha)},k^{-d}\}\log k.$$

Furthermore, multiplying
$u_k-\bar u$
in both sides of the equality \eqref{e:proof1} and taking integral with respect to
$\mu^{k}$, we can derive that for all $k>k_1(\w)$,
\begin{align*}
 \lambda\|u_k-\bar u\|^2_{L^2(k^{-1}\Z^d;\mu^{k})}
&\le  \lambda\|u_k-\bar u\|^2_{L^2(k^{-1}\Z^d;\mu^{k})}+
\E^{k,\w}(u_k-\bar u, u_k-\bar u)\\
&\le \int_{k^{-1}\Z^d} |J_1^k(x)||u_k(x)-\bar u(x)| \,\mu^{k}(dx)\\
&\le \left(\int_{k^{-1}\Z^d} |J_1^k(x)|^2\,\mu^{k}(dx)\right)^{1/2}\|u_k-\bar u\|_{L^2(k^{-1}\Z^d;\mu^{k})}\\
&\le  c_{43}\left(\max\{k^{-2(1-\alpha)},k^{-d}\}\log k \right)^{1/2}\|u_k-\bar u\|_{L^2(k^{-1}\Z^d;\mu^{k})},\end{align*} which proves that
$$\|u_k-\bar u\|_{L^2(k^{-1}\Z^d;\mu^{k})}\le c_{44}\max\{k^{-(1-\alpha)},k^{-d/2}\}\log^{1/2}k.$$
Therefore, for all $k>k_1(\w)$,
\begin{equation*}
\begin{split}
 \|u_k-\bar u\|^2_{L^2(\R^d;dx)}
&=\sum_{z\in k^{-1}\Z^d}\int_{\prod_{1\le i\le d}(z_i,z_i+k^{-1}]}|u_k(x)-\bar u(x)|^2\,dx\\
&\le 2\sum_{z\in k^{-1}\Z^d}\int_{\prod_{1\le i\le d}(z_i,z_i+k^{-1}]}\left(|u_k(x)-\bar u(z)|^2+
|\bar u(x)-\bar u(z)|^2\right)\,dx\\
&=2\|u_k-\bar u\|^2_{L^2(k^{-1}\Z^d;\mu^{k})}+2\sum_{z\in k^{-1}\Z^d}\int_{\prod_{1\le i\le d}(z_i,z_i+k^{-1}]}
|\bar u(x)-\bar u(z)|^2\,dx\\
&\le 2\|u_k-\bar u\|^2_{L^2(k^{-1}\Z^d;\mu^{k})}+c_{45}k^{-2}\le c_{46} \max\{k^{-2(1-\alpha)},k^{-d}\}\log k,
\end{split}
\end{equation*}
where in the second inequality we used \eqref{e:pppqqq}. Combining this with \eqref{t3-1-11}, we obtain
the desired conclusion \eqref{t3-1-1-1} for $\alpha\in (0,1)$.
\end{proof}

\appendix
\section{Dense property of  $\mathscr{S}_0^\lambda$}\label{appendix1}

Recall that for any $\lambda>0$,
\begin{equation}\label{e:S}
\mathscr{S}_0^\lambda := \left\{f:  f= (\lambda - \bar \sL) g \hbox{ for some } g\in   C_c^\infty(\R^d) \right\},
\end{equation}
where $\bar \sL$ is defined by \eqref{e:1.2}.
\begin{lemma}\label{l5-1}
$\mathscr{S}_0^\lambda \subset C_b(\R^d)\cap L^2(\R^d;dx)$ and is dense in $L^2(\R^d;dx)$.
\end{lemma}

\begin{proof}
First note that $\bar \sL$ is a constant multiple of the fractional Laplacian $\Delta^{\alpha}$ so its $\lambda$-resolvent $\bar G_\lambda:= (\lambda - \bar \sL)^{-1}$
has a symmetric kernel $\bar G_\lambda (z)$ so that $\bar G_\lambda f(x)= \int_{\R^d} \bar G_\lambda (x-y) f(y) \,dy= \int_{\R^d} \bar G_\lambda (y) f(x-y) \,dy$ for any $f\in L^2(\R^d; dx)$.
Thus $\bar G_\lambda f \in C^\infty_b (\R^d)$ for any $f\in C^\infty_c (\R^d)$ with
$\nabla^k \bar G_\lambda f  = \bar G_\lambda (\nabla^k f)$
for  $k\in \N$.
For any $g\in C^2_b(\R^d)$, since
\begin{equation}\label{e:A1}
\bar \sL g(x) =  \int_{ \R^d }  \left(g(y) -g(x) -\I_{\{|y-x|\leq 1\}} \nabla g (x) \cdot   |y-x| \right) \frac1{|x-y|^{d+\alpha}} \,dy \quad \hbox{for } x\in \R^d,
\end{equation}
it holds that $\bar \sL g \in C_b (\R^d)$
and  there is a constant $c_0 >0$ independent of $g$
such that
\begin{equation}\label{e:A2}
 |\bar \sL g(x)|\le c_0  (
 \| \nabla^2 g \|_\infty
 + \| g\|_\infty) \quad \hbox{for every } x\in \R^d.
 \end{equation}
 For $g\in C^2_c (\R^d)$, let
 $R>1$ so that ${\rm supp} [g] \subset B_R$.
  We have by \eqref{e:A1}-\eqref{e:A2} that there are constants $c_1,  c_2>0$ so that
$$
|(\lambda -\bar \sL ) g (x) |\le  \lambda |g (x)|+ |\bar \sL g   (x) |
\le   c_1\I_{B_{2R}}(x)+ \frac{c_2}{(1+|x|)^{d+\alpha}}\I_{B_{2R}^c}(x) \quad \hbox{for } x\in \R^d.
$$
Thus $(\lambda -\bar \sL ) g  \in C_b(\R^d) \cap L^2(\R^d; \mu)$, proving that $\mathscr{S}_0^\lambda \subset C_b(\R^d)\cap L^2(\R^d;dx)$.

Next we show that $\mathscr{S}_0^\lambda$ is dense in $L^2(\R^d; \mu)$.
 Let $\varphi\in C_b^\infty(\R)$ be such that $0\le \varphi \le 1$ on $\R$ with  $\varphi(r)=1$ for
 $r\le 1$ and $\varphi(r)=0$ for $r\ge 2$.
    For any $f\in C^\infty_c(\R^d)$ and $n\ge1$,  define
$ g_n(x):=  \varphi (|x |-n  ) \bar G_\lambda f (x)$,
 which is  in $C_c^\infty (\R^d)$.
  Let    $f_n:= (\lambda -\bar \sL) g_n  \in \mathscr{S}_0^\lambda$.
We  have
 \begin{align*}
\lim_{n \to \infty}\|f_n-f\|_{L^2(\R^d;dx)}
&=\lim_{n \to \infty}\|(\lambda-\bar \sL)
((\varphi(|\cdot| -n )-1)\bar G_\lambda f(\cdot))\|_{L^2(\R^d;dx)}\\
&\le \lambda \lim_{n \to \infty}\|
((\varphi(|\cdot| -n )-1) \bar G_\lambda f(\cdot))\|_{L^2(\R^d;dx)}\\
&\quad+ \lim_{n \to \infty}\|\bar \sL
((\varphi(|\cdot| -n )-1)
\bar G_\lambda f(\cdot))
\|_{L^2(\R^d;dx)}\\
&=0,
\end{align*}
where in the last  equality we used the dominated convergence theorem, and the facts that $\bar G_\lambda f \in L^2(\R^d; \mu)$,
\eqref{e:A2}
and the bound
$$|(\varphi(|x| -n)-1)\bar G_\lambda f(x)| + |\nabla^2((\varphi(|x| -n )-1)\bar G_\lambda f(x)) | \le \frac{c_3}{(1+|x|)^{d+\alpha}}\I_{\{|x|\ge n\}} \quad
\hbox{on } \R^d.
$$
Thus we have shown that any $f\in C^\infty_c(\R^d)$ can be $L^2(\R^d; dx)$-approximated by functions in $\mathscr{S}_0^\lambda$.
As $C^\infty_c(\R^d)$ is dense in $L^2(\R^d; dx)$,  this  establishes that  $\mathscr{S}_0^\lambda$ is dense in $L^2(\R^d;dx)$.
\end{proof}

\ \

\noindent {\bf Acknowledgements.}\,\,  The research of Xin Chen is supported by National Natural Science Foundation of China
(Nos. 12122111).
The research of Zhen-Qing Chen is partially supported by  a Simons Foundation fund.
The research of Takashi Kumagai is supported by JSPS
KAKENHI Grant Number
JP22H00099.
The research
of Jian Wang is supported by the National Key R\&D Program of China (2022YFA1000033) and the National Natural Science Foundation of China (Nos. 11831014, 12071076 and 12225104).

\end{document}